\documentclass[10pt,twoside]{article}

\usepackage{etoolbox}
\csundef{leftbar}
\csundef{endleftbar}

\usepackage[utf8]{inputenc}
\usepackage[T1]{fontenc}
\usepackage[french,english]{babel} 
\usepackage{lmodern}
\usepackage{makeidx} 
\usepackage{graphicx}

\usepackage{caption}
\usepackage{eurosym}
\usepackage{xfrac}
\usepackage{listings}
\usepackage{mathrsfs}
\usepackage{euscript}
\usepackage{stmaryrd}
\usepackage{eqnarray}
\usepackage{multirow}
\usepackage{bold-extra}
\usepackage{blkarray}

\usepackage{amsmath}
\usepackage{amsthm}
\usepackage{keytheorems}
\newkeytheoremstyle{defn}{
	spaceabove=\topsep, spacebelow=\topsep,%
	headfont=\normalfont\bfseries,%
	notefont=\normalfont, notebraces={(}{)},%
	bodyfont=\normalfont,%
	postheadspace=.5em,%
}
\newkeytheorem{defi}[numberwithin=section,name=Definition,style=defn]
\newkeytheorem{definosec}[name=Definition,style=defn]
\newkeytheorem{quesnosec}[name=Question,style=defn]
\newkeytheorem{defifr}[numberwithin=section,name=D\'efinition,style=defn]
\newkeytheorem{exe}[sibling=defi,name=Example,style=defn]
\newkeytheorem{exenosec}[sibling=definosec,name=Example,style=defn]
\newkeytheorem{exefr}[sibling=defifr,name=Exemple,style=defn]
\newkeytheorem{exes}[sibling=defi,name=Examples,style=defn]
\newkeytheorem{const}[sibling=defi,name=Construction,style=defn]
\newkeytheorem{constfr}[sibling=defifr,name=Construction,style=defn]
\newkeytheorem{defin}[numbered=no,name=D\'efinition,style=defn]
\newkeytheorem{conj}[numbered=no,name=Conjecture,style=defn]
\newkeytheorem{nconj}[numbered=no,name=Naive conjecture,style=defn]
\newkeytheorem{con}[sibling=defi,name=Conjecture,style=defn]
\newkeytheorem{confr}[sibling=defi,name=Conjecture,style=defn]
\newkeytheoremstyle{proposition}{
	spaceabove=\topsep, spacebelow=\topsep,
	headfont=\normalfont\bfseries,
	notefont=\normalfont, notebraces={(}{)},
	bodyfont=\itshape,
	postheadspace=.5em,
	qed=
}
\newkeytheorem{prop}[sibling=defi,name=Proposition,style=proposition]
\newkeytheorem{propnosec}[sibling=definosec,name=Proposition,style=proposition]
\newkeytheorem{propfr}[sibling=defifr,name=Proposition,style=proposition]
\newkeytheorem{propo}[numbered=no,name=Proposition,style=proposition]
\newkeytheorem{theo}[sibling=defi,name=Theorem,style=proposition]
\newkeytheorem{theonn}[numbered=no,name=Theorem,style=proposition]
\newkeytheorem{theonosec}[sibling=definosec,name=Theorem,style=proposition]
\newkeytheorem{theofr}[sibling=defifr,name=Th\'eor\`eme,style=proposition]
\newkeytheorem{theoreme}[numbered=no,name=Th\'eor\`eme,style=proposition]
\newkeytheorem{lem}[sibling=defi,name=Lemma,style=proposition]
\newkeytheorem{lemnosec}[sibling=definosec,name=Lemma,style=proposition]
\newkeytheorem{lemfr}[sibling=defifr,name=Lemme,style=proposition]
\newkeytheorem{cor}[sibling=defi,name=Corollary,style=proposition]
\newkeytheorem{cornosec}[sibling=definosec,name=Corollary,style=proposition]
\newkeytheorem{cornn}[numbered=no,name=Corollary,style=proposition]
\newkeytheorem{corfr}[sibling=defifr,name=Corollaire,style=proposition]
\newkeytheoremstyle{remarque}{
	spaceabove=\topsep, spacebelow=\topsep,
	headfont=\normalfont\itshape,
	notefont=\normalfont, notebraces={(}{)},
	bodyfont=\normalfont,
	postheadspace=.5em,
	qed=
}
\newkeytheorem{rema}[sibling=defi,name=Remark,style=remarque]
\newkeytheorem{remafr}[sibling=defifr,name=Remarque,style=remarque]
\newkeytheorem{reman}[numbered=no,name=Remark,style=remarque]
\newkeytheorem{remanfr}[numbered=no,name=Remarque,style=remarque]
\newkeytheorem{remas}[sibling=defi,name=Remarks,style=remarque]
\newkeytheorem{nota}[sibling=defi,name=Notation,style=remarque]
\newkeytheorem{notafr}[sibling=defifr,name=Notation,style=remarque]
\newkeytheorem{term}[sibling=defi,name=Terminology,style=remarque]
\newkeytheorem{termfr}[sibling=defifr,name=Terminologie,style=remarque]
\newkeytheorem{remo}[numbered=no,name=Remorse,style=remarque]

\usepackage{amssymb}
\usepackage{amstext}

\makeatletter
\DeclareFontFamily{OMX}{MnSymbolE}{}
\DeclareSymbolFont{MnLargeSymbols}{OMX}{MnSymbolE}{m}{n}
\SetSymbolFont{MnLargeSymbols}{bold}{OMX}{MnSymbolE}{b}{n}
\DeclareFontShape{OMX}{MnSymbolE}{m}{n}{
    <-6>  MnSymbolE5
   <6-7>  MnSymbolE6
   <7-8>  MnSymbolE7
   <8-9>  MnSymbolE8
   <9-10> MnSymbolE9
  <10-12> MnSymbolE10
  <12->   MnSymbolE12
}{}
\DeclareFontShape{OMX}{MnSymbolE}{b}{n}{
    <-6>  MnSymbolE-Bold5
   <6-7>  MnSymbolE-Bold6
   <7-8>  MnSymbolE-Bold7
   <8-9>  MnSymbolE-Bold8
   <9-10> MnSymbolE-Bold9
  <10-12> MnSymbolE-Bold10
  <12->   MnSymbolE-Bold12
}{}

\let\llangle\@undefined
\let\rrangle\@undefined
\DeclareMathDelimiter{\llangle}{\mathopen}%
                     {MnLargeSymbols}{'164}{MnLargeSymbols}{'164}
\DeclareMathDelimiter{\rrangle}{\mathclose}%
                     {MnLargeSymbols}{'171}{MnLargeSymbols}{'171}
\makeatother

\usepackage{tikz}
\usetikzlibrary{arrows,matrix,cd,babel}
\usepackage{quiver}

\usepackage{bbm}
\usepackage{multicol}

\usepackage{import}

\usepackage{lettrine}

  {\vspace*{\fill}\clearpage\restoregeometry}
  
\usepackage[french,lined]{algorithm2e}

\newcommand{\Zb}{\mathbb{Z}}
\newcommand{\Rb}{\mathbb{R}}

\newcommand{\Cb}{\mathbb{C}}

\newcommand{\Pb}{\mathbb{P}}
\newcommand{\Abb}{\mathbb{A}}

\newcommand{\Er}{\mathrm{E}}

\newcommand{\m}{\mathfrak{m}}

\newcommand{\Kr}{\mathrm{K}}

\newcommand{\GW}{\mathrm{GW}}
\newcommand{\MW}{\mathrm{MW}}

\newcommand{\Nis}{\mathrm{Nis}}

\newcommand{\Sr}{\mathrm{S}}

\newcommand{\CH}{\mathrm{CH}}
\newcommand{\Hr}{\mathrm{H}}

\newcommand{\Kbf}{\mathbf{K}}
\newcommand{\Ibf}{\mathbf{I}}

\newcommand{\Fr}{\mathrm{F}}
\newcommand{\Br}{\mathrm{B}}

\newcommand{\Wr}{\mathrm{W}}
\newcommand{\et}{\text{\'et}}

\newcommand{\Mr}{\mathrm{M}}
\newcommand{\Wbf}{\mathbf{W}}
\newcommand{\Gr}{\mathrm{G}}

\newcommand{\Ch}{\mathrm{Ch}}

\newcommand{\Jbf}{\mathbf{J}}

\newcommand{\Abf}{\mathbf{A}}

\newcommand{\alg}{\mathrm{alg}}
\newcommand{\Osc}{\mathscr{O}}

\newcommand{\Lc}{\mathcal{L}}

\newcommand{\Gm}{\mathbb{G}_m}

\newcommand{\Hsc}{\mathscr{H}}

\newcommand{\GL}{\mathrm{GL}}

\newcommand{\Vsc}{\mathscr{V}}

\newcommand{\Ceu}{\EuScript{C}}

\newcommand{\Ec}{\mathcal{E}}
\newcommand{\BGL}{\mathrm{BGL}}

\newcommand{\Gra}{\mathrm{Gr}}
\newcommand{\BGm}{\mathrm{B}\mathbb{G}_m}
\newcommand{\GWb}{\mathbf{GW}}
\newcommand{\Tbf}{\mathbf{T}}
\newcommand{\Sbf}{\mathbf{S}}

\DeclareMathOperator{\Ker}{Ker}

\DeclareMathOperator{\Spec}{Spec}

\DeclareMathOperator{\Coker}{Coker}

\DeclareMathOperator{\rk}{rk}

\DeclareMathOperator{\Pic}{Pic}

\DeclareMathOperator{\Sq}{Sq}

\renewcommand{\Im}{\operatorname{Im}}

\makeatletter 

\newcommand*\Neginternal[3]{\mathpalette\Neg@{{#1}{#2}{#3}}}
\newcommand*\Neg@[2]{\Neg@@{#1}#2}
\newcommand*\Neg@@[4]{%
  \mathrel{\ooalign{%
    $\m@th#1#4$\cr
    \hidewidth$\m@th#3{#1}\mkern\muexpr#2*2$\hidewidth\cr
  }}%
}

\newcommand*\negslash[1]{\m@th#1\not\mathrel{\phantom{=}}}
\newcommand*\snegslash[1]{\rotatebox[origin=c]{60}{$\m@th#1-$}}
\newcommand*\ssnegslash[1]{\rotatebox[origin=c]{60}{$\m@th#1{\dabar@}\mkern-7mu{\dabar@}$}}
\newcommand*\sssnegslash[1]{\rotatebox[origin=c]{60}{$\m@th#1\dabar@$}}
\makeatother  

\usepackage[newparttoc,pagestyles,toctitles]{titlesec}
\usepackage{titletoc}

\renewcommand*\thesection{\arabic{section}}
\titleformat{\section}[block]{\Large\scshape\filcenter}{\thesection.}{10pt}{}[]

\newpagestyle{main}{

\sethead[][\thepage][]
	{}{\thepage}{}
\setfoot[][][]
	{}{}{}}
	
\usepackage{minitoc}	
\mtcselectlanguage{english}	
\doparttoc

\usepackage{appendix}

 \usepackage[style=english]{csquotes}

\usepackage[backend=biber,style=ext-alphabetic, maxbibnames=50]{biblatex}
\DefineBibliographyExtras{english}{\DeclarePunctuationPairs{comma}{*?!}}

\DeclareBibliographyDriver{article}{%
  \usebibmacro{bibindex}%
  \usebibmacro{begentry}%
  \usebibmacro{author}%
  \setunit{\addcomma\space}
  \printfield[emph]{title}
  \setunit{\addcomma\space}
  \printfield{journaltitle}
  \setunit{\addspace}%
  \textbf{\printfield{volume}}
  \setunit{\addspace}%
  \printfield[parens]{year}
  \setunit{\addcomma\space}
  \printfield{number}
  \setunit{\addcomma\space}
  \printfield{pages}
  \usebibmacro{finentry}%
}

\DeclareFieldFormat[article]{title}{#1}
\DeclareFieldFormat[article]{journaltitle}{#1}
\DeclareFieldFormat[article]{pages}{#1}
\DeclareFieldFormat[article]{number}{\bibstring{number}~#1}

\DeclareBibliographyDriver{book}{%
  \usebibmacro{bibindex}%
  \usebibmacro{begentry}%
  \printnames{author}
  \setunit{\addcomma\space}
  \printfield{title}
  \setunit{\addcomma\space}
  \printfield{edition}
  \setunit{\addcomma\space}
  \printfield{series}
  \setunit{\addcomma\space}
  \printfield{volume}
  \setunit{\addcomma\space}
  \printfield{number}
  \setunit{\addcomma\space}
  \printlist{publisher}
  \setunit{\addcomma\space}
  \printlist{location}
  \setunit{\addcomma\space}
  \printfield{year}
  \usebibmacro{finentry}%
}

\DeclareFieldFormat[book]{volume}{\bibstring{volume}~#1}
\DeclareFieldFormat[book]{number}{\bibstring{number}~#1}
\DeclareFieldFormat[book]{publisher}{#1}

\DeclareBibliographyDriver{incollection}{%
  \usebibmacro{bibindex}%
  \usebibmacro{begentry}%
  \printnames{author}
  \setunit{\addcomma\space}
  \printfield{title}
  \setunit{\addperiod\space}
  \usebibmacro{in:}
  \printfield[emph]{booktitle}
  \setunit{\addcomma\space}
  \usebibmacro{byeditor+others}
  \setunit{\addcomma\space}
  \printfield{edition}
  \setunit{\addcomma\space}
  \printfield{series}
  \setunit{\addcomma\space}
  \printfield{volume}
  \setunit{\addcomma\space}
  \printfield{number}
  \setunit{\addcomma\space}
  \printlist{publisher}
  \setunit{\addcomma\space}
  \printlist{location}
  \setunit{\addcomma\space}
  \printfield{year}
  \setunit{\addcomma\space}
  \printfield{pages}
  \usebibmacro{finentry}%
}

\DeclareFieldFormat[incollection]{volume}{\bibstring{volume}~#1}
\DeclareFieldFormat[incollection]{number}{\bibstring{number}~#1}
\DeclareFieldFormat[incollection]{series}{#1,\addspace}

\DefineBibliographyExtras{english}{%
}%

\usepackage{authblk}

\newcommand\imCMsym[4][\mathord]{%
  \DeclareFontFamily{U} {#2}{}
  \DeclareFontShape{U}{#2}{m}{n}{
    <-6> #25
    <6-7> #26
    <7-8> #27
    <8-9> #28
    <9-10> #29
    <10-12> #210
    <12-> #212}{}
  \DeclareSymbolFont{CM#2} {U} {#2}{m}{n}
  \DeclareMathSymbol{#4}{#1}{CM#2}{#3}
}

\imCMsym{cmmi}{124}{\CMjmath}
\imCMsym[\mathop]{cmsy}{113}{\CMamalg}
\imCMsym[\mathop]{cmex}{96}{\CMcoprod}

\numberwithin{equation}{section}

\usepackage[colorlinks=true]{hyperref}

\DeclareMathAlphabet{\mathcal}{OMS}{cmsy}{m}{n}

\usepackage[a4paper,top=1.05in, bottom=1.05in, left=1.05in, right=1.05in]{geometry}

\begin{document}

\selectlanguage{english}

\pagestyle{main}

\title{On the cohomological classification of vector bundles on smooth real affine surfaces and threefolds}
\author{Samuel Lerbet}
\affil{DMA, École normale supérieure, Université PSL, CNRS, 75005 Paris, France}
\date{\today}

\maketitle

\begin{abstract}
We study the cohomological classification of vector bundles on smooth real affine surfaces and threefolds. We show that, as was observed in joint work with A. Asok and J. Fasel and with S. Banerjee and J. Fasel, under suitable cohomological assumptions on the real locus of such varieties, this classification mirrors the one obtained on algebraically closed base fields by Mohan Kumar--Murthy and Asok--Fasel. Using an argument due to Fasel, we also give an efficient proof of a theorem of Kucharz characterising the triples of algebraic cycles that can be realised as the Chern classes of a rank $3$ bundle on a smooth real affine threefold. We further answer the questions left open by Kucharz; to our knowledge, we give the first instance of a projective module over a smooth affine $\Rb$-algebra of dimension $3$ with trivial Chern classes which is not stably free.
\end{abstract}

\section{Introduction}

\paragraph*{Cohomological classification of vector bundles.} If $X$ is a smooth variety over a field $k$ (by variety over $k$, we mean separated $k$-scheme of finite type), we denote by $\CH^p(X)$ the Chow group of codimension $p$ cycles on $X$ modulo rational equivalence and we set $\Ch^p(X)=\CH^p(X)/2$. If $E$ is a vector bundle on $X$, the Chow groups of $X$ house the \emph{Chern classes} $c_p(E)\in\CH^p(X)$ of $E$ defined using, \emph{e.g.}, the axiomatic treatment of \cite{grothendieckTheorieClassesChern1958}. The Chern classes of $E$ detect the rank in the sense that $c_p(E)=0$ if $p>r$ where $r$ is the rank of $E$. Thus denoting by $\Vsc_r(X)$ the collection of isomorphism classes of rank~$r$ vector bundles on $X$, there is a well-defined map \[\varphi_r(X)=(c_1,\ldots,c_r):\Vsc_r(X)\to\prod_{i=1}^r\CH^i(X)\] taking the isomorphism class $\{E\}$ of $E$ to $(c_i(E))_i$.

With a view towards the cohomological classification of vector bundles on smooth affine varieties, one might ask two questions about these maps.
\begin{itemize}
	\item What is the image of $\varphi_r(X)$? Namely, what $r$-tuples of algebraic cycles on $X$ can be realised as the tuple of Chern classes of a rank $r$ vector bundle on $X$?
	\item What are the fibres of $\varphi_r(X)$? For example, under what conditions on $X$ are rank $r$ vector bundles on $X$ determined up to isomorphism by their Chern classes?
\end{itemize}
From now on in this introduction, we will focus on the case where $X$ is a smooth affine threefold. In this situation, when the base field $k$ is algebraically closed, these questions have very satisfactory answers. Indeed, it follows from \cite{kumarAlgebraicCyclesVector1982} and \cite{asokCohomologicalClassificationVector2014} that $\varphi_3(X)$ and $\varphi_2(X)$ are bijections in this case. Moreover, Suslin's cancellation theorem \cite{suslinCancellationTheoremProjective1977} and Serre's splitting theorem \cite[Théorème 1]{serreModulesProjectifsEspaces1957} imply that the stabilisation map $\Vsc_3(X)\to\Vsc_r(X)$ carrying $E$ to $E\oplus\Osc_X^{r-3}$ is bijective for every $r\geqslant 3$ so by stability of Chern classes, namely since $c_*(E)=c_*(E\oplus\Osc_X)$ for every vector bundle $E$, the map $\varphi_r(X)$ is bijective for every $r\geqslant 3$.

Moving now to the case where $k=\Rb$ is the field of real numbers, the situation is more delicate. For instance, consider the real algebraic $3$-sphere $\Sr_\Rb^3=\Spec(\Rb[x,y,z,t]/\langle x^2+y^2+z^2+t^2-1\rangle)$. Since $\Sr_\Rb^3$ is affine, by \cite[Theorem 1.3]{colliot-theleneZerocyclesCohomologyReal1996}, the top Chow group $\CH^3(\Sr_\Rb^3)$ contains as a direct summand a copy of $(\Zb/2)^t$ where $t$ is the number of compact connected components of the real locus $\Sr_\Rb^3(\Rb)$ of $\Sr_\Rb^3$: this real locus is the usual $3$-dimensional sphere in $\Rb^4$ which is certainly compact so $\Zb/2$ injects into $\CH^3(\Sr_\Rb^3)$. In particular, this latter group is nontrivial. On the other hand, by \cite{faselProjectiveModulesReal2011}, every algebraic vector bundle on $\Sr_\Rb^3$ is trivial so the image of $\varphi_3(\Sr_\Rb^3)$ is the triple $(0,0,0)$ of Chern classes of $\Osc_X^3$. Therefore $\varphi_3(\Sr_\Rb^3)$ is not surjective. Moveover, the map $\varphi_2(X)$ is not injective in general: indeed, the tangent bundle $T$ of $\Sr_\Rb^2$ is stably trivial so its Chern classes are trivial, but it is not trivial because of the hairy ball theorem, so $T$ pulls back to a rank $2$ vector bundle $T_{|X}$ on $X=\Sr_\Rb^2\times\Abb^1$ such that $\{T_{|X}\}\neq\{\Osc_X^2\}$ but $\varphi_2(\{T_{|X}\})=(0,0)$. Altogether, this motivates the study of the above questions for threefolds over~$\Rb$. The image of $\varphi_3$ was investigated by Kucharz in \cite{kucharzAlgebraicCyclesVector1988} where the following theorem is proved.

\begin{theonn}[Kucharz]
Let $X$ be a smooth real affine threefold and let $(c_1,c_2,c_3)\in\CH^1(X)\times\CH^2(X)\times\CH^3(X)$. Then denoting by $w_i$ the image of $c_i$ under the mod $2$ cycle class map $\overline{\gamma}^i:\CH^i(X)\to\Hr^i(X(\Rb),\Zb/2)$, the triple $(c_1,c_2,c_3)$ lies in the image of $\varphi_3(X)$ if, and only if, the relation $w_3-w_1\cup w_2-\Sq(w_2)=0$ holds in $\Hr^3(X(\Rb),\Zb/2)$ where $\Sq$ is the first Steenrod square operation.
\end{theonn}

Kucharz then asks whether $\varphi_3(X)$ is in fact injective in the situation of the above theorem. Moving to rank $2$ bundles, given a pair $(c_1,c_2)\in\CH^1(X)\times\CH^2(X)$ where $X$ is a smooth real affine threefold and setting $w_i=\overline{\gamma}^i(c_i)\in\Hr^i(X(\Rb),\Zb/2)$, if there exists a rank~$2$ vector bundle $E$ on $X$ with $c_i(E)=c_i$, then the rank $3$ bundle $E\oplus\Osc_X$ has Chern classes $(c_1,c_2,0)$ hence $\Sq(w_2)=w_1\cup w_2$ in $\Hr^3(X(\Rb),\Zb/2)$ according to the above theorem. Letting $S(X)$ denote the collection of pairs in $\CH^1(X)\times\CH^2(X)$ such that this relation is satisfied, the inclusion $\varphi_2(X)\subseteq S(X)$ thus holds. Kucharz then asks if this inclusion is an equality. For ease of reference, let us record these problems; in the following questions, the letter $X$ denotes a smooth real affine threefold.
\begin{itemize}
	\item[(A)] Is $\varphi_3(X)$ injective?
	\item[(B)] Is the inclusion $\Im\varphi_2(X)\subseteq S(X)$ an equality?
\end{itemize}

Our goal in this note is to study the above theorem, Questions (A) and (B) and related questions on the cohomological classification of vector bundles. In fact, a quick analysis of the Moore--Postnikov tower of the relevant morphism of motivic spaces in Kucharz's theorem shows that on a smooth affine threefold $X$ (over any perfect field with $2$ invertible), there is exactly one obstruction $o(c)$, living in $\Ch^3(X)$, to realising a triple $(c_i)_i$ of classes in Chow groups as the Chern classes of a rank $3$ bundle. Should this obstruction be equal over any field to $\overline{c_3}-\overline{c_1}\cdot\overline{c_2}-\Sq(\overline{c_2})$ where $\Sq$ is the first Steenrod square of Voevodsky \cite{voevodskyMotivicCohomology2coefficients2003} or Brosnan \cite{brosnanSteenrodOperationsChow2003}, the results of \cite{colliot-theleneZerocyclesCohomologyReal1996} would then imply Kucharz's theorem. This identification of $o(c)$, which would yield a generalisation of Kucharz's result to a claim over arbitrary fields, was part of the motivation for our work. In any case, over $\Rb$, our first main result is a positive answer to Question (B):

\begin{theonn}[Theorem \ref{theo:question_B_kucharz}]
Let $X$ be a smooth real affine threefold and let $(c_1,c_2)\in\CH^1(X)\times\CH^2(X)$ be a pair of algebraic cycles; set $w_i=\overline{\gamma}^i(X)$. Then there exists a rank $2$ vector bundle $E$ on $X$ whose Chern classes satisfy $c_i(E)=c_i$ if, and only if, the relation $\Sq(w_2)+w_1\cup w_2=0$ holds in $\Hr^3(X(\Rb),\Zb/2)$.
\end{theonn}

The proof uses motivic obstruction theory, and in particular the Postnikov tower of $\BGL_2$ which was investigated in detail in \cite{asokOBSTRUCTIONSALGEBRAIZINGTOPOLOGICAL2019} to study the obstructions to algebraising rank $2$ complex vector bundles on a smooth complex affine fourfold, and results from \cite{lerbetImageHigherSignature2026} on the cohomology of real algebraic varieties. The situation is less pleasant for Question (A) which, as a consequence of the following theorem, admits a negative answer.

\begin{theonn}[Proposition \ref{prop:ce_qu_A}]
There exists a smooth real affine threefold $U$ and a rank $2$ vector bundle $E$ on $U$ such that $c_i(E)=0$ for $i=\{1,2\}$ but $E$ is \emph{not} stably trivial: for every $n\geqslant 0$, the bundle $E\oplus\Osc_U^n$ is nontrivial.
\end{theonn}

In particular, if $U$ and $E$ are as in this theorem, then the rank $3$ bundles $E\oplus\Osc_U$ and $\Osc_U^3$ both have trivial Chern classes, but are not isomorphic. Note that the usual example of a nontrivial bundle with trivial Chern classes, namely the tangent bundle to $\Sr_\Rb^2$, becomes trivial after addition of a free rank $1$ bundle (the normal bundle of the inclusion $\Sr_\Rb^2\hookrightarrow\Abb_\Rb^3$) and thus does not yield an example of rank $3$ as in the above theorem (working implicitly over $\Sr_\Rb^2\times\Abb^1$ to obtain a threefold). To our knowledge, the above theorem gives the first example of a projective module over a smooth real affine algebra of dimension $3$ whose Chern classes are trivial, but which is not stably free. The smooth affine variety $U$ and the vector bundle $E$ used in the proof of the above theorem were constructed in \cite{asokSplittingVectorBundles2025}. In particular, the (reduced) $K$-theory class $[E]$ of $E$ can be deduced from the computations performed in that paper, and we show by explicit calculations in $K$-theory that $[E]$ is nonzero.

In \cite{kucharzAlgebraicCyclesVector1988}, to motivate Question (A), Kucharz observes that the answer to this question is positive if the extension of scalars $X_\Cb=X\times_\Rb\Spec\Cb$ of $X$ to $\Cb$ is rational and $X(\Rb)$ has no compact connected component. Although the answer to Question (A) is negative in general, in accordance with the philosophy developed in \cite{asokSplittingVectorBundles2025,banerjeeSuslinsCancellationConjecture2026} we show in this paper that Kucharz's topological condition is in fact the crucial one.

\begin{theonn}[Example \ref{exe:classification_no_compact_connected_component_surface}, Theorem \ref{theo:classification_by_chern_classes}]
Let $d\leqslant 3$ and let $X$ be a smooth affine $\Rb$-variety of dimension~$d$. If $X(\Rb)$ has no compact connected component, then $\varphi_d(X)$ is injective.
\end{theonn}

To compare this result to \cite{asokSplittingVectorBundles2025,banerjeeSuslinsCancellationConjecture2026}, note that the singular cohomology group $\Hr^d(X(\Rb),\Zb/2)$ is the $\Zb/2$-vector space generated by the compact connected components of $X(\Rb)$ so the assumption that $X(\Rb)$ has no compact connected component is equivalent to the vanishing of $\Hr^d(X(\Rb),\Zb/2)$. Hence it expresses that the real locus of $X$ is ``cohomologically small'', and is precisely the kind of hypotheses that appear in \cite[Theorems 5.1.1 and 5.2.1]{asokSplittingVectorBundles2025} and that are introduced in \cite[§2.1]{banerjeeSuslinsCancellationConjecture2026}.

\paragraph*{Contents.} The paper is organized as follows. Section \ref{sec:preliminaries} is preliminary and collects notations used in the article. In Section \ref{sec:question_B}, we study the cohomological classification of rank $2$ bundles on smooth real affine surfaces, revisiting and generalising the main result of \cite{bargeFibresAlgebriquesSurface1987} and using results of \cite{asokSplittingVectorBundles2025} to slightly generalise their arguments. We then answer Question (B) positively and we investigate analogues of the cohomological classification obtained in \cite{asokCohomologicalClassificationVector2014} for rank $2$ bundles on smooth real affine threefolds with small real locus; as noted in Remark \ref{rema:hypotheses_classification_rank_2}, they are essentially optimal. In Section \ref{sec:rank_3_bundles}, we present a proof of a generalisation of Kucharz's result over arbitrary fields communicated to us by Fasel. We then show that the answer to Question (A) is negative in general. We further observe that the injectivity of $\varphi_3(X)$ is intimately related to the topology of the real locus by proving that $\varphi_3(X)$ is injective if $X$ is a smooth real threefold such that $X(\Rb)$ has no compact connected component, removing the rationality condition from Kucharz's observation in this direction.

\paragraph*{Acknowledgements.} We thank Jean Fasel for very useful discussions and for providing us with the proof of Theorem \ref{theo:Kucharz_over_a_field}, and Olivier Benoist for encouraging us to write up these results. The author was supported by ANR project CYCLADES, grant number ANR-23-CE40-0011, during his work on this project.

\section{Preliminaries}\label{sec:preliminaries}

Unless otherwise specified, by $\Hr^*(X,\Abf)$, we denote the Nisnevich cohomology of a Noetherian scheme $X$ with coefficients in the sheaf $\Abf$ on the small Nisnevich site $X_\Nis$ of $X$ (we usually write that $\Abf$ is a sheaf on $X$ by abuse of terminology). We recall that $X$ is smooth over a field, then the Nisnevich cohomological dimension of $X$ is bounded above by its Krull dimension $\dim(X)$ so $\Hr^j(X,\Abf)=0$ for $j>\dim(X)$ for every sheaf $\Abf$ on $X_\Nis$.

If $X$ is a smooth variety over a field, we endow the $\Kr$-theory group $\Kr_0(X)$ of $X$ with the filtration $\Fr^\bullet\Kr_0(X)$ given by the codimension of the support: the group $\Fr^p\Kr_0(X)$ consists of those classes $\alpha\in\Kr_0(X)$ such that there exists a closed subscheme $Z$ of $X$ of codimension $\geqslant p$ such that $\alpha_{|X\setminus Z}=0$. We denote the graded pieces of this filtration by $\Gra^p\Kr_0(X)=\Fr^p\Kr_0(X)/\Fr^{p+1}\Kr_0(X)$.

We use the Steenrod square $\Sq$ defined by Voevodsky in \cite{voevodskyMotivicCohomology2coefficients2003} on Chow groups mod~$2$. Given a smooth variety $X$ over a field, if $\Lc\in\Pic(X)$ has mod $2$ first Chern class $\overline{c_1}$, we denote by $\Sq_\Lc$ or $\Sq_{\overline{c_1}}$ the first Steenrod square $\Ch^*(X)\to\Ch^{*+1}(X)$ twisted by $\Lc$, defined by $\Sq_\Lc=\Sq_{\overline{c_1}}=\Sq+\overline{c_1}\cdot(\text{--})$. We use the same notation in topology: if $L$ is a real topological line bundle on a smooth manifold and $w_1$ is its first Stiefel--Whitney class, we set $\Sq_L=\Sq_{w_1}=\Sq+w_1\cup(\text{--})$.

We denote by $\Kbf_*^\MW$ the unramified Milnor--Witt $\Kr$-theory sheaf of $\Zb$-graded rings, by $\Kbf_*^\Mr$ the unramified Milnor $\Kr$-theory sheaf and by $\Ibf^*$ the graded sheaf of powers of the fundamental ideal; we refer to \cite[Chapter 3]{morelA1AlgebraicTopologyField2012} for the definition of these sheaves. If $X$ is a smooth variety and $\Lc$ is a line bundle on $X$, these sheaves can be twisted by $\Lc$ into sheaves $\Kbf_*^\MW(\Lc)$ and $\Ibf^*(\Lc)$ on $X$ (twisting does not change the sheaf $\Kbf_*^\Mr$) and there is an exact sequence \[0\to\Ibf^{*+1}(\Lc)\to\Kbf_*^\MW(\Lc)\to\Kbf_*^\Mr\to 0\] of sheaves on $X$. We set $\widetilde{\CH}^c(X,\Lc)=\Hr^c(X,\Kbf_c^\MW(\Lc))$ and note that $\Hr^c(X,\Kbf_c^\Mr)=\CH^c(X)$.

If $X$ is a smooth variety over $\Rb$, we endow its real locus $X(\Rb)$ with the Euclidean topology for which it underlies a smooth manifold; cohomology of $X(\Rb)$ is then sheaf cohomology. We denote the mod~$2$ cycle class map from Chow groups to mod~$2$ cohomology by $\overline{\gamma}^i:\CH^i(X)\to\Hr^i(X(\Rb),\Zb/2)$ and we use the same notation for its factorisation by $\Ch^i(X)$. Steenrod operations on Chow groups mod~$2$ and mod~$2$ cohomology are compatible: if $\Lc$ is a line bundle on $X$ with associated real topological line bundle $L=\Lc(\Rb)$ on $X(\Rb)$, then $\Sq_L\circ\overline{\gamma}^i=\overline{\gamma}^{i+1}\circ\Sq_\Lc$ (see for instance \cite[Lemma 4.1.1]{asokSplittingVectorBundles2025}). Over $\Rb$, the groups $\Hr^*(X,\Ibf^\star(\Lc))$ and $\Hr^*(X,\Kbf_\star^\MW(\Lc))$ support \emph{quadratic real cycle class maps} \[\gamma_\star^*:\Hr^*(X,\Ibf^\star(\Lc))\to\Hr^*(X(\Rb),\Zb(\Lc(\Rb))),\;\widetilde{\gamma}_\star^*:\Hr^*(X,\Kbf_\star^\MW(\Lc))\to\Hr^*(X(\Rb),\Zb(\Lc(\Rb)))\] described, \emph{e.g.}, in \cite[§2.5]{lerbetImageHigherSignature2026} and which are compatible with the maps $\overline{\gamma}^*$ \cite[diagram before Lemma 2.29]{lerbetImageHigherSignature2026}.

\section{Rank 2 vector bundles}\label{sec:question_B}

\subsection{Motivic obstruction theory}

We use the formalism of Moore--Postnikov towers in motivic homotopy theory as laid out, \emph{e.g.}, in \cite[§6.1]{asokSplittingVectorBundles2014}. We refer to this work for the notions of Eilenberg--Mac Lane spaces, (homotopy) fibres and pullbacks and so on. If $X$ is a presheaf of spaces, we let $X_+$ denote $X$ with a disjoint added base point; if $X$ is any pointed presheaf, we let $\pi_n(X)$ denote the $n$-th homotopy sheaf of its image under the localisation functor from presheaves of spaces to motivic spaces.

We write $\BGL_r$ for the motivic localisation of the Nisnevich classifying space $\Br_\Nis\GL_r$ for the sheaf of groups $\GL_r$. In this section, we denote by $(\BGL_2^{(n)})_{n\geqslant 1}$ the Postnikov tower of $\BGL_2$. The $k_1$-invariant $\BGL_2\to\BGL_2^{(1)}=\Br\pi_1(\BGL_2)=\BGm$ induces a map $\det:[X,\BGL_2]_{\Abb^1}\to[X,\BGm]_{\Abb^1}$ for any pointed motivic space $X$. If $U$ is a smooth affine $k$-scheme, then by the affine $\Abb^1$-representability of vector bundles (\cite{morelA1AlgebraicTopologyField2012,schlichtingEulerClassGroups2017,asokAffineRepresentabilityResults2017}), there is a pointed bijection $\Vsc_2(U)\cong[U_+,\BGL_2]_{\Abb^1}$ modulo which $\det:[U_+,\BGL_2]_{\Abb^1}\to[U_+,\BGm]_{\Abb^1}\cong\Pic(U)$ is the determinant map. Since the canonical map $i_2:\BGL_2\to\BGL_2^{(2)}$ is $2$-connected, the $k_1$-invariant also induces a map $\det:[U_+,\BGL_2^{(2)}]_{\Abb^1}\to\Pic(U)$. Moreover, by \cite[Proposition 6.3]{asokCohomologicalClassificationVector2014}, if $\Lc$ is a line bundle on $X$, then the $k_2$-invariant induces a bijection \[[U_+,\BGL_2^{(2)}]_{\Abb^1}\supseteq\mathrm{det}^{-1}(\{\Lc\})\xrightarrow[\cong]{(k_2)_*}\widetilde{\CH}^2(U,\Lc).\] The triangle
\begin{equation}\label{eq:classification_rank_2_bundles}
\begin{tikzcd}
	{\Vsc_2(U)\cong[U_+,\BGL_2]_{\Abb^1}} && {[U_+,\BGL_2^{(2)}]_{\Abb^1}} \\
	& {\Pic(U)}
	\arrow["{(i_2)_*}", from=1-1, to=1-3]
	\arrow["\det"', from=1-1, to=2-2]
	\arrow["\det", from=1-3, to=2-2]
\end{tikzcd}
\end{equation}
is then commutative since both diagonal maps are induced by $k_1$-invariants. Letting $\Vsc_2(U,\Lc)$ denote the fibre of $\det:\Vsc_2(U)\to\Pic(U)$ over the isomorphism class of $\Lc$, the induced map $\Vsc_2(U,\Lc)\to[U_+,\BGL_2^{(2)}]_{\Abb^1}\xrightarrow{(k_2)_*}\widetilde{\CH}^2(U,\Lc)$ on fibres over the isomorphism class of $\Lc$ is Morel's Euler class $e$.

\subsection{On surfaces}\label{subsection:rank_2_bundles_on_surfaces}

Let $X$ be a smooth affine surface over a perfect field $k$ of characteristic not $2$. Then since $X$ has dimension $2$, by \cite[Proposition 6.2]{asokCohomologicalClassificationVector2014}, the canonical map $(i_2)_*:\Vsc_2(X)\cong[X_+,\BGL_2]_{\Abb^1}\to[X_+,\BGL_2^{(2)}]_{\Abb^1}$ is bijective. This bijection is compatible with determinant maps in view of the commutative triangle in (\ref{eq:classification_rank_2_bundles}) so the Euler class $e:\Vsc_2(X,\Lc)\to\widetilde{\CH}^2(U,\Lc)$ is a bijection for every line bundle $\Lc$ on $X$.

\begin{lem}\label{lem:surjectivity_rank_2_surfaces}
Let $X$ be a smooth affine surface over a perfect field $k$ of characteristic not $2$. Then $\varphi_2(X):\Vsc_2(X)\to\CH^1(X)\times\CH^2(X)$ is surjective.
\end{lem}

\begin{proof}
Let $\Lc$ be a line bundle on $X$. The cokernel of the comparison map $\widetilde{\CH}^2(X,\Lc)\to\CH^2(X)$ is a subgroup of $\Hr^3(X,\Ibf^3(\Lc))$ which vanishes for cohomological dimension reasons since $X$ is a surface so this comparison map is surjective. The composite \[\Vsc_2(X,\Lc)\xrightarrow{e}\widetilde{\CH}^2(X,\Lc)\to\CH^2(X)\] is then the second Chern class by \cite[Remark 7.22]{morelA1AlgebraicTopologyField2012} so that $c_2:\Vsc_2(X,\Lc)\to\CH^2(X)$ is surjective on each fibre over $\Pic(X)$. This is a reformulation of the surjectivity of $\varphi_2(X)$.
\end{proof}

If $E$ is a rank $2$ vector bundle on $X$, then $E$ underlies an alternating form $\varphi_E:E\otimes E\to\det E$ carrying $x\otimes y$ to $x\wedge y$. This gives rise to a map 
\begin{equation}\label{eq:unstable_to_stable}
\Phi:\Vsc_2(X,\Lc)\to\widetilde{\GW}^2(X,\Lc)
\end{equation}
taking $\{E\}$ to $[E,\varphi_E]-[\Hr_\Lc(\Osc)]$, where $\GW^2(X,\Lc)$ is the Grothendieck--Witt group of alternating forms on $X$, the subgroup $\widetilde{\GW}^2(X,\Lc)\subseteq\GW^2(X,\Lc)$ consists of the rank $0$ forms and $\Hr_\Lc(\Osc_X)$ is the $\Lc$-valued symplectic hyperbolic form on $\Osc_X$.

\begin{lem}\label{lem:unstable_to_stable_rank_2_surfaces}
The map $\Phi:\Vsc_2(X,\Lc)\to\widetilde{\GW}^2(X,\Lc)$ is a bijection.
\end{lem}

\begin{proof}
This can be checked directly, but we give an argument that will make the link between the results of this subsection and those of \cite{bargeFibresAlgebriquesSurface1987} clear. The Gersten--Grothendieck--Witt spectral sequence \[\Er(2)_1^{p,q}=\bigoplus_{x\in X^{(p)}}\GW_{2-p-q}^{2-p}(\kappa(x),\omega_{x/X}\otimes\Lc(x))\Rightarrow\GW_{2-(p+q)}^2(X,\Lc)\] (see \cite[§3.1]{faselChowWittGroups2009}) converges to the groups $\GW_{2-*}^2(X,\Lc)$ filtered by codimension of the support. The arguments of \cite[Example 1.2.5]{asokSplittingVectorBundles2025} show that the edge homomorphism $\widetilde{\GW}^2(X,\Lc)\to\Er(2)_2^{2,0}=\widetilde{\CH}^2(X,\Lc)$ is an isomorphism. It is known to coincide with the first Borel class $b_1$ of \cite{paninMotivicCommutativeRing2019}, hence the composite \[\Vsc_2(X,\Lc)\to\widetilde{\GW}^2(X,\Lc)\to\widetilde{\CH}^2(X,\Lc)\] is the (Chow--Witt) Euler class $e_{cw}$. On the other hand, Morel's Euler class, which coincides with $e_{cw}$ by \cite{asokComparingEulerClasses2016}, induces a bijection $e:\Vsc_2(X,\Lc)\to\widetilde{\CH}^2(X,\Lc)$ as already discussed. The Euler class and the first Borel class are then known to coincide for rank $2$ bundles so the triangle
\[\begin{tikzcd}
	{\Vsc_2(X,\Lc)} && {\widetilde{\GW}^2(X,\Lc)} \\
	& {\widetilde{\CH}^2(X,\Lc)}
	\arrow["\Phi", from=1-1, to=1-3]
	\arrow["e"', from=1-1, to=2-2]
	\arrow["{b_1}", from=1-3, to=2-2]
\end{tikzcd}\]
commutes. The diagonal maps $e$ and $b_1$ in the above diagram are bijections hence the top horizontal map is also bijective as required.
\end{proof}

Let $X$ be a smooth real affine surface. The group $\widetilde{\GW}^2(X)$ (with twist $\Lc=\Osc_X$) coincides with the reduced symplectic $\Kr$-theory group $\widetilde{\Kr}_0^{\mathrm{Sp}}(X)$ described in \cite[§6]{bargeFibresAlgebriquesSurface1987} as an extension of the Witt group $\Wr^2(X)$ of the exact category of alternating forms on $X$ (in the sense of \cite{balmerTriangularWittGroups2001}) by a quotient of the reduced $\Kr$-theory group $\widetilde{\Kr}_0(X)$. Barge and Ojanguren show that the alternating part $\Wr^2(X)$ is determined by essentially topological data on the real locus $X(\Rb)$ (this is strictly true modulo $2$-torsion). More precisely, they prove the following theorem (\cite[Théorème 6.2]{bargeFibresAlgebriquesSurface1987}).

\begin{theo}[Barge--Ojanguren]\label{theo:W2_surface}
Let $X$ be a smooth variety of dimension $2$ over $\Rb$. Let $\Lc$ be a line bundle on $X$ and denote by $L$ the associated real topological line bundle on $X(\Rb)$. Let $d_L:\Hr^1(X(\Rb),\Zb/2)\to\Hr^2(X(\Rb),\Zb(L))$ be the Bockstein homomorphism, that is, the connecting homomorphism for the cohomology long exact sequence determined by the epimorphism $\Zb(L)\to\Zb/2$ of sheaves. If $X$ is not proper or $X(\Rb)$ is not empty, then there is a canonical isomorphism $\Wr^2(X,\Lc)\cong\Hr^2(X(\Rb),\Zb(L))/d_L(\Hr_\alg^1(X(\Rb),\Zb/2))$ where $\Hr_\alg^1(X(\Rb),\Zb/2)$ is the image of the mod $2$ cycle class map $\overline{\gamma}^1:\CH^1(X)\to\Hr^1(X(\Rb),\Zb/2)$. If $X$ is proper and $X(\Rb)=\emptyset$, then $\Wr^2(X,\Lc)=\Coker(\Sq_\Lc:\Ch^1(X)\to\Ch^2(X))$.
\end{theo}

\begin{rema}
In \cite{bargeFibresAlgebriquesSurface1987}, the above theorem is only stated for $X$ affine and $\Lc=\Osc_X$ (in particular, the final assertion does not appear). Using the results of \cite{lerbetImageHigherSignature2026}, it is not difficult to prove the full statement written above so we shall give the proof in this generality (the isomorphism constructed in the proof below in the affine case is essentially the same as Barge--Ojanguren's).
\end{rema}

\begin{proof}
Since $X$ is regular of dimension $\leqslant 3$, by the Gersten--Witt spectral sequence \cite{balmerGerstenWittSpectral2002}, there is an isomorphism $\Wr^2(X,\Lc)\cong\Hr^2(X,\Wbf(\Lc))$. The cohomology of the sheaves $\Wbf(\Lc)$ and $\Ibf(\Lc)$ are computed by Rost--Schmid complexes \cite[Chapter 5]{morelA1AlgebraicTopologyField2012} defined in terms of the contractions $\Ibf_{-1}$ and $\Wbf_{-1}$ of the sheaves $\Ibf$ and $\Wbf$ in the sense of, \emph{e.g.}, \cite[§2.3]{asokSplittingVectorBundles2014}. Since $\Ibf_{-1}\cong\Wbf_{-1}=\Wbf$, the inclusion $\Ibf\hookrightarrow\Wbf$ induces an isomorphism on Rost--Schmid complexes in degree $\geqslant 1$ thus an isomorphism $\Hr^1(X,\Ibf(\Lc))\cong\Hr^2(X,\Wbf(\Lc))$. Thus it suffices to prove that $\Hr^2(X,\Ibf(\Lc))$ admits a description as in the statement of Theorem \ref{theo:W2_surface}. For this, we use the quadratic real cycle class maps for the sheaves $\Ibf^t(\Lc)$ which induce a commutative square
\[\begin{tikzcd}
	{\Hr^1(X,\overline{\Ibf}^1)} & {\Hr^2(X,\Ibf^2(\Lc))} \\
	{\Hr^1(X(\Rb),\Zb/2)} & {\Hr^2(X(\Rb),\Zb(L))}
	\arrow["{\partial_\Lc}", from=1-1, to=1-2]
	\arrow["{\overline{\gamma}^1}"', from=1-1, to=2-1]
	\arrow["{\gamma^2}"', from=1-2, to=2-2]
	\arrow["{d_L}"', from=2-1, to=2-2]
\end{tikzcd}\]
where $\overline{\Ibf}^1=\Ibf(\Lc)/\Ibf^2(\Lc)$ (see for instance \cite[ladder before (2-4)]{lerbetImageHigherSignature2026}). The affirmation of the Milnor conjectures yields an isomorphism $\overline{\Ibf}^1\cong\Hsc^1$ of sheaves where $\Hsc^1$ is the sheaf on $X$ associated with the presheaf $U\mapsto\Hr_\et^1(U,\Zb/2)$. Then the exact sequence \[0\to\Ibf^2(\Lc)\to\Ibf(\Lc)\to\overline{\Ibf}\to 0\] of sheaves on $X$ induces an exact sequence \[\Hr^1(X,\Hsc^1)\xrightarrow{\partial_\Lc}\Hr^2(X,\Ibf^2(\Lc))\to\Hr^2(X,\Ibf(\Lc))\to\Hr^2(X,\Hsc^1).\] By the Bloch--Ogus theorem \cite{blochGerstensConjectureHomology1974}, one has $\Hr^1(X,\Hsc^1)=\Ch^1(X)$ and $\Hr^2(X,\Hsc^1)=0$. Therefore $\Hr^2(X,\Ibf(\Lc))$ is the cokernel of $\partial_\Lc$. Moreover, modulo the isomorphism $\Hr^1(X,\overline{\Ibf}^1)\cong\Ch^1(X)$, the map $\Hr^1(X,\overline{\Ibf}^1)\to\Hr^1(X(\Rb),\Zb/2)$ is $\overline{\gamma}^1$: in particular, it has image $\Hr_\alg^1(X(\Rb),\Zb/2)$. Thus the map $\gamma^2:\Hr^2(X,\Ibf^2(\Lc))\to\Hr^2(X(\Rb),\Zb(L))$ induces a homomorphism \[\chi:\Hr^2(X,\Ibf(\Lc))\to\Hr^2(X,\Ibf(\Lc))/d_L\Hr_\alg^1(X(\Rb),\Zb/2)\] on cokernels. But since $X$ is not proper or $X(\Rb)$ is not empty, according to \cite[Theorem 3.5]{lerbetImageHigherSignature2026}, the map $\gamma^2$ is an isomorphism. Thus so is $\chi$, so that there is an isomorphism \[\Wr^2(X,\Lc)\cong\Hr^2(X,\Wbf(\Lc))\cong\Hr^2(X,\Ibf(\Lc))\xrightarrow[\cong]{\chi}\Hr^2(X(\Rb),\Zb(L))/d_L\Hr_\alg^1(X(\Rb),\Zb/2),\] as required. 

Suppose now that $X$ is proper and that $X(\Rb)$ is empty. In this case, the exact top line in the previous diagram shows that $\Hr^2(X,\Ibf(\Lc))\cong\Wr^2(X,\Lc)$ is isomorphic to $\Coker\partial_\Lc$. Since $X$ has dimension $2$, by Jacobson's theorem \cite[Theorem 8.11]{jacobsonRealCohomologyPowers2017}, the map $\gamma_3^2:\Hr^2(X,\Ibf^3(\Lc))\to\Hr^2(X(\Rb),\Zb(L))=0$ is an isomorphism so $\Hr^2(X,\Ibf^3(\Lc))=0$. The sheaf $\Ibf^3(\Lc)$ is the kernel of the epimorphism $\Ibf^2(\Lc)\to\overline{\Ibf}^2$ so the group $\Hr^2(X,\Ibf^3(\Lc))=0$ surjects onto the kernel of the quotient map $\pi:\Hr^2(X,\Ibf^2(\Lc))\to\Hr^2(X,\overline{\Ibf}^2)$. It follows that $\pi$ is injective; since $\Coker\pi$ is a subgroup of $\Hr^3(X,\Ibf^3(\Lc))$, which vanishes because $X$ has dimension $2$, the map $\pi$ is also surjective, hence an isomorphism. Modulo the isomorphisms $\Hr^1(X,\overline{\Ibf}^1)\cong\Ch^1(X)$ and $\Hr^2(X,\overline{\Ibf}^2)\cong\Ch^2(X)$ provided by the affirmation of the Milnor conjecture, the composite \[\Hr^1(X,\overline{\Ibf}^1)\xrightarrow{\partial_\Lc}\Hr^2(X,\Ibf^2(\Lc))\xrightarrow{\pi}\Hr^2(X,\overline{\Ibf}^2)\] is precisely $\Sq_\Lc$ by \cite[Theorem 3.4.1]{asokSecondaryCharacteristicClasses2015} and \cite[Theorem 1.1]{totaroNONINJECTIVITYMAPWITT2003}. It follows that $\pi$ induces an isomorphism $\Wr^2(X,\Lc)=\Coker\partial_\Lc\cong\Coker\Sq_\Lc$. The proof is therefore complete.
\end{proof}

\begin{rema}
This argument can readily be generalized to give a description of $\Hr^d(X,\Wbf(\Lc))$ if $X$ is a smooth $\Rb$-variety of dimension $d$ and thus of the shifted Witt group $\Wr^d(X,\Lc)$ for small $d$. This is explained in greater detail in \cite[§3]{lerbetWittGroupsSmooth2026}.
\end{rema}

Barge--Ojanguren's result thus gives an essentially complete description of $\Vsc_2(X,\Lc)$ in terms of $\widetilde{\Kr}_0(X)$ and the cohomology of $X(\Rb)$ (given the knowledge of the algebraic part of its mod $2$ cohomology). A different description can be given using \cite[Proposition 2.2.5]{asokSplittingVectorBundles2025}. Recall the statement of this proposition in the case of smooth affine varieties:

\begin{prop}[\protect{\cite[Proposition 2.2.5]{asokSplittingVectorBundles2025}}]\label{prop:cartesian_square_chow--witt}
Let $X$ be a smooth real affine variety of dimension $d\geqslant 2$ and let $\Lc$ be a line bundle on $X$; set $L=\Lc(\Rb)$. Then the square
\[\begin{tikzcd}
	{\widetilde{\CH}^d(X,\Lc)} & {\CH^d(X)} \\
	{\Hr^d(X(\Rb),\Zb(L))} & {\Hr^d(X(\Rb),\Zb/2)}
	\arrow[from=1-1, to=1-2]
	\arrow["{\widetilde{\gamma}^d}"', from=1-1, to=2-1]
	\arrow["{\overline{\gamma}^d}", from=1-2, to=2-2]
	\arrow["{\mathrm{mod}\;2}"', from=2-1, to=2-2]
\end{tikzcd}\]
is cartesian.
\end{prop}

Proposition \ref{prop:cartesian_square_chow--witt} also easily implies the following theorem.

\begin{theo}\label{theo:classification_rank_2_bundles_on_surfaces}
Let $E$ and $E'$ be rank $2$ bundles on a smooth real affine surface $X$. The following assertions are then equivalent.
\begin{itemize}
	\item[(i)] The bundles $E$ and $E'$ are isomorphic.
	\item[(ii)] There is an equality $c_*(E)=c_*(E')$ of Chern classes and the topological vector bundles $E(\Rb)$ and $E'(\Rb)$ are isomorphic.\footnote{Here we do not require that this isomorphism come from an algebraic morphism between $E$ and $E'$.}
	\item[(iii)] There is an equality $c_*(E')=c_*(E')$ of Chern classes and an equality $e(E(\Rb))=e(E'(\Rb))$ of topological Euler classes.
\end{itemize}
\vspace{-\topsep}\vspace{-\topsep}
\end{theo}

\begin{proof}
It is clear that (i) implies (ii) and (ii) implies (iii). Suppose that (iii) holds. In particular, one has $c_1(E)=c_1(E')$ so $E$ and $E'$ have isomorphic determinants. Let $\Lc=\det E$, so that $E$ and $E'$ determine elements $\{E\}$ and $\{E'\}$ of $\Vsc_2(X,\Lc)$. To conclude, it suffices to prove that $e(E)=e(E')$ in $\widetilde{\CH}^2(X,\Lc)$. According to Proposition \ref{prop:cartesian_square_chow--witt}, there is an injection $\widetilde{\CH}^2(X,\Lc)\to\CH^2(X)\times\Hr^2(X(\Rb),\Zb(L))$; as noted in the proof of \cite[Theorem 2.1.1]{asokSplittingVectorBundles2025}, this injection carries $e(E)$ to $(c_2(E),e(E(\Rb)))$. Thus by (iii), one has $e(E)=e(E')$ as required.
\end{proof}

\begin{exe}\label{exe:classification_no_compact_connected_component_surface}
Let $\Lc$ be a line bundle on a smooth real affine surface $X$ and set $L=\Lc(\Rb)$. Denote by $\Ceu(L)$ the set of compact connected components $V$ such that $L_{|V}$ is isomorphic to the orientation sheaf of $V$. If $\Ceu(L)$ is empty, for example, if $\Lc=\Osc_X$ and no compact connected component of $X$ is orientable, \emph{e.g.}, if $X(\Rb)$ has no compact connected component, then the map $\Hr^2(X(\Rb),\Zb(L))\to\Hr^2(X(\Rb),\Zb/2)$ is an isomorphism hence so is the map $\widetilde{\CH}^2(X,\Lc)\to\CH^2(X)$ by pullback. It follows that rank $2$ bundles $E$ and $E'$ on $X$ with determinant isomorphic to $\Lc$ are isomorphic if, and only if, there is an equality $c_2(E)=c_2(E')$ of Chern classes. In particular, if $\Hr^2(X(\Rb),\Zb/2)=0$, so that $X(\Rb)$ has no compact connected component, then the map $\varphi_2(X):\Vsc_2(X)\to\CH^1(X)\times\CH^2(X)$ is a bijection. 

Conversely, if $X(\Rb)$ has a compact connected component, then the group $\Hr^2(X(\Rb),\Zb(\omega_{X(\Rb)}))$ (where $\omega_{X(\Rb)}$ is the orientation sheaf of $X(\Rb)$) contains $\Zb$ as a direct summand so the kernel of the reduction mod $2$ map $\Hr^2(X(\Rb),\Zb(\omega_{X/\Rb}))\to\Hr^2(X(\Rb),\Zb/2)$ contains $2\Zb$. Thus denoting by $\omega_{X/\Rb}$ the canonical sheaf of the smooth variety $X$, the kernel of the map $\widetilde{\CH}^2(X,\omega_{X/\Rb})\to\CH^2(X)$ also contains $2\Zb$ and the rank $2$ bundles of determinant equal to $\omega_{X/\Rb}$ in $\Pic(X)/2$ are not classified by their second Chern class.
\end{exe}

\begin{exe}\label{exe:algebraic_vector_bundles_on_real_two_sphere}
Let $X$ be a smooth affine surface over $\Rb$ and suppose that $\CH^2(X_\Cb)=0$. Since $\CH^2(X_\Cb)$ is divisible by \cite[Lemma 1.2]{colliot-theleneZerocyclesCohomologyReal1996}, this assumption is satisfied, \emph{e.g.}, if $X_\Cb$ is rational as $\CH^2(X_\Cb)$ is then also a quotient of $\Zb$. Let $p:X_\Cb\to X$ denote the projection. According to \cite[Theorem 1.3]{colliot-theleneZerocyclesCohomologyReal1996}, there is an exact sequence \[\CH^2(X_\Cb)\xrightarrow{p_*}\CH^2(X)\to(\Zb/2)^t\to 0\] of abelian groups where $t$ is the number of compact connected components of $X(\Rb)$. Since $\CH^2(X_\Cb)=0$, the map $\CH^2(X)\to(\Zb/2)^t$ is an isomorphism. In particular, the group $\CH^2(X)$ is $2$-torsion and thus coincides with $\Ch^2(X)$. On other hand, the homomorphism $\overline{\gamma}^2:\Ch^2(X)\to\Hr^2(X(\Rb),\Zb/2)$ is an isomorphism by \cite[Theorem 3.2 (d)]{colliot-theleneZerocyclesCohomologyReal1996}. We conclude that $\overline{\gamma}^2:\CH^2(X)\to\Hr^2(X(\Rb),\Zb/2)$ is an isomorphism: by pullback and using Proposition \ref{prop:cartesian_square_chow--witt}, we see that the map $\widetilde{\gamma}^2:\widetilde{\CH}^2(X,\Lc)\to\Hr^2(X(\Rb),\Zb(\Lc(\Rb)))$ is an isomorphism for every line bundle $\Lc$ on $X$. Theorem \ref{theo:classification_rank_2_bundles_on_surfaces} can then be rephrased in the following way: if $E$ and $E'$ are rank $2$ bundles on $X$, then $E\simeq E'$ if, and only if, one has $c_1(E)=c_1(E')$ and the topological bundles $E(\Rb)$ and $E'(\Rb)$ over $X(\Rb)$ are isomorphic.

For example, assume that $X=\Sr_\Rb^2$. Then $\Sr_\Rb^2$ is rational (over $\Rb$) by stereographic projection. Moreover, the Picard group of $\Sr_\Rb^2$ is well-known to be trivial so equality of first Chern classes is automatic. We conclude if $E$ and $E'$ are rank $2$ bundles on $\Sr_\Rb^2$, then $E\simeq E'$ if, and only if, the topological bundles $E(\Rb)$ and $E'(\Rb)$ are isomorphic. This recovers \cite[Proposition 7.5]{bargeFibresAlgebriquesSurface1987} (though of course, the methods are quite similar).
\end{exe}

\subsection{On threefolds with small real locus}

\subsubsection*{Question (B)}

We review some of the results of \cite{asokOBSTRUCTIONSALGEBRAIZINGTOPOLOGICAL2019}. Consider the map \[\varphi_
2:\BGL_2\to\Kr(\Kbf_1^\Mr,1)\times\Kr(\Kbf_2^\Mr,2)=B\] of motivic spaces induced by Chern classes. If $U$ is a smooth affine scheme, then the composite \[\Vsc_2(U)\cong[U_+,\BGL_2]_{\Abb^1}\xrightarrow{(\varphi_2)_*}\Hr^1(U,\Kbf_1^\Mr)\times\Hr^2(U,\Kbf_2^\Mr)=\CH^1(U)\times\CH^2(U)\] is the map $\varphi_2(U)$, hence the notation. Let $F_2$ be the fibre of $\varphi_2$. Since $\pi_3(\Kr(\Kbf_1^\Mr,1)\times\Kr(\Kbf_2^\Mr,2))=1$ and $F_2$ is simply connected so that $\pi_1(F_2)=1$, the long exact sequence of homotopy sheaves associated to the previous fibration sequence reads \[1\to\pi_2(F_2)\to\pi_2(\BGL_2)=\Kbf_2^\MW\to\pi_2(\Kr(\Kbf_1^\Mr,1)\times\Kr(\Kbf_2^\Mr,2))=\Kbf_2^\Mr\to 1.\] The map $c_1:\pi_1(\BGL_2)\to\Kbf_1^\Mr=\Gm$ is the determinant map which is an isomorphism (\cite{asokCohomologicalClassificationVector2014}). Then \cite[Proposition 2.2.1]{asokOBSTRUCTIONSALGEBRAIZINGTOPOLOGICAL2019} implies that the inclusion $\pi_2(F)\to\pi_2(\BGL_2)$ identifies $\pi_2(F)$ with $\Ibf^3$ \emph{as a sheaf with additive action of the multiplicative group $\Gm$}.

We are now in a position to answer Question (B).

\begin{theo}\label{theo:question_B_kucharz}
Let $X$ be a smooth affine threefold over $\Rb$ and let $(c_1,c_2)\in\CH^1(X)\times\CH^2(X)$; set $w_i=\overline{\gamma}^i(c_i)\in\Hr^i(X(\Rb),\Zb/2)$. The following assertions are then equivalent.
\begin{enumerate}
	\item The pair $(c_1,c_2)$ lies in $\Im\varphi_2(X)$.
	\item The equality $\Sq_{w_1}(w_2)=0$ holds in $\Hr^3(X(\Rb),\Zb/2)$.
\end{enumerate}
\vspace{-\topsep}\vspace{-\topsep}
\end{theo}

\begin{proof}
The pair $(c_1,c_2)$ induces a pointed morphism $f:X_+\to B$. Then $(c_1,c_2)$ lies in $\Im\varphi_2(X)$ if, and only if, there exists a lift of $f$ through $\varphi_2$. The first nontrivial stage of the Moore--Postnikov tower of the map $\varphi_2$ is the second stage $\Ec_2$ because the homotopy fibre $F_2$ is simply connected. Since $X$ has dimension $3$, according to \cite[Corollary B.5]{morelA1AlgebraicTopologyField2012}, the underlying map $i_2:\BGL_2\to\Ec_2$ induces a surjection $(i_2)_*:[X_+,\BGL_2]_{\Abb^1}\to[X_+,\Ec_2]_{\Abb^1}$. Therefore $(c_1,c_2)$ lies in $\Im\varphi_2(X)$ if, and only if, there exists a lift of $f$ through the map $q:\Ec_2\to B$. This map sits in a pullback square
\[\begin{tikzcd}
	{\Ec_2} & {B\Gm} \\
	{B} & {\Kr^{\Gm}(\pi_2(F),3)}
	\arrow[from=1-1, to=1-2]
	\arrow["{q}"', from=1-1, to=2-1]
	\arrow[from=1-2, to=2-2]
	\arrow["{k_2}"', from=2-1, to=2-2]
\end{tikzcd}\] hence the obstruction to lifting $f$ along $q$ is the cohomology class in $\Hr^3(X,\pi_2(F_3)(\Lc))=\Hr^3(X,\Ibf^3(\Lc))$ corresponding to $k_2\circ f$, where $\Lc$ is the line bundle on $X$ corresponding to $c_1\in\CH^1(X)\cong\Pic(X)$. The cohomology long exact sequence associated with the inclusion $\Ibf^3(\Lc)\subseteq\Kbf_2^\MW(\Lc)$ gives a connecting homomorphism $\widetilde{\partial}_\Lc:\Hr^2(X,\Kbf_2^\Mr)\to\Hr^3(X,\Ibf^3(\Lc))$ which is in fact induced by the $k$-invariant $k_2$ as explained in the proof of \cite[Theorem 2.2.2]{asokOBSTRUCTIONSALGEBRAIZINGTOPOLOGICAL2019}. This identifies the obstruction class $o(f)\in\Hr^2(X,\pi_2(F_2)(\Lc))$, determined by the composite $k_2\circ f$, with $\widetilde{\partial}_\Lc(c_2)$. We conclude that $(c_1,c_2)$ lies in $\Im\varphi_2(X)$ if, and only if, the class $\widetilde{\partial}_\Lc(c_2)\in\Hr^3(X,\Ibf^3(\Lc))$ vanishes.

Now we consider the commutative ladder
\[\begin{tikzcd}
	0 & {\Ibf^3(\Lc)} & {\Kbf_2^\MW(\Lc)} & {\Kbf_2^\Mr} & 0 \\
	0 & {\Ibf^3(\Lc)} & {\Ibf^2(\Lc)} & {\overline{\Ibf}^2} & 0
	\arrow[from=1-1, to=1-2]	
	\arrow[from=1-2, to=1-3]
	\arrow["{=}"', from=1-2, to=2-2]
	\arrow[from=1-3, to=1-4]
	\arrow[from=1-3, to=2-3]
	\arrow["\lrcorner"{anchor=center, pos=0.125}, draw=none, from=1-3, to=2-4]
	\arrow[from=1-4, to=2-4]
	\arrow[from=1-4, to=1-5]
	\arrow[from=2-1, to=2-2]
	\arrow[from=2-2, to=2-3]
	\arrow[from=2-3, to=2-4]
	\arrow[from=2-4, to=2-5]
\end{tikzcd}\]
deduced from the description of Milnor--Witt $\Kr$-theory as a fibre product obtained in \cite[Théorème 5.3]{morelPuissancesLidealFondamental2004}. This commutative ladder induces a commutative square
\[\begin{tikzcd}
	{\Hr^2(X,\Kbf_2^\Mr)} & {\Hr^3(X,\Ibf^3(\Lc))} \\
	{\Hr^2(X,\overline{\Ibf}^2)} & {\Hr^3(X,\Ibf^3(\Lc))}
	\arrow["{\widetilde{\partial}_\Lc}", from=1-1, to=1-2]
	\arrow[from=1-1, to=2-1]
	\arrow["{=}", from=1-2, to=2-2]
	\arrow["{\partial_\Lc}"', from=2-1, to=2-2]
\end{tikzcd}\]
The affirmation of the Milnor conjecture yields an isomorphism $\Hr^2(X,\overline{\Ibf}^2)\cong\Ch^2(X)$ modulo which the homomorphism \[\CH^2(X)=\Hr^2(X,\Kbf_2^\Mr)\to\Hr^2(X,\overline{\Ibf}^2)=\Ch^2(X)\] is reduction mod $2$; we let $\overline{c_2}$ be the image of $c_2$ in $\Hr^2(X,\overline{\Ibf}^2)$. Then $\widetilde{\partial}_\Lc(c_2)=0$ if, and only if, the cohomology class $\partial_\Lc(\overline{c_2})$ vanishes.

To understand this vanishing, we use the commutative ladder
\[\begin{tikzcd}
	{\Hr^2(X,\overline{\Ibf}^2)} & {\Hr^3(X,\Ibf^3(\Lc))} & {\Hr^3(X,\overline{\Ibf}^3)} \\
	{\Hr^2(X(\Rb),\Zb/2)} & {\Hr^3(X(\Rb),\Zb(L))} & {\Hr^3(X(\Rb),\Zb/2)}
	\arrow["{\partial_\Lc}", from=1-1, to=1-2]
	\arrow["{\overline{\gamma}^2}"', from=1-1, to=2-1]
	\arrow[from=1-2, to=1-3]
	\arrow["\gamma^3"', from=1-2, to=2-2]
	\arrow["{\overline{\gamma}^3}"', from=1-3, to=2-3]
	\arrow["{d_L}"', from=2-1, to=2-2]
	\arrow["\rho"', from=2-2, to=2-3]
\end{tikzcd}\]
studied in \cite[proof of Lemma 4.1.1]{asokSplittingVectorBundles2025} where $L=\Lc(\Rb)$ is the topological line bundle associated with $L$, the map $d_L$ is the connecting homomorphism for the cohomology long exact sequence associated with the epimorphism $\Zb(L)\to\Zb/2$ of sheaves and $\rho$ is reduction mod $2$ of the coefficients. The middle vertical map $\gamma^3$ is an isomorphism by \cite[Theorem 3.5]{lerbetImageHigherSignature2026} so by commutation of the left square in the above ladder, the equality $\partial_\Lc(\overline{c_2})=0$ holds if, and only if, the image of $\overline{\gamma}^2(\overline{c_2})=w_2$ under $d_L$ vanishes. Now by Poincaré duality, the group $\Hr^3(X(\Rb),\Zb(L))$ is isomorphic to \[\bigoplus_{\Ceu(L)}\Zb\oplus\bigoplus_{\Ceu\setminus\Ceu(L)}\Zb/2\] where $\Ceu$ is the collection of compact connected components of $X(\Rb)$ and $\Ceu(L)\subseteq\Ceu$ is the subset of those components $V$ such that $L_{|V}$ is isomorphic to the orientation sheaf of $V$ (see for example \cite[Corollary 1.2.2]{faselVasersteinSymbolReal2018}). Since $\Hr^2(X(\Rb),\Zb/2)$ is a $2$-torsion group, the image of $d_L$ is contained in $\bigoplus_{\Ceu\setminus\Ceu(L)}\Zb/2$, on which $\rho$ is injective. Therefore $d_L(w_2)=0$ if, and only if, the image of $w_2$ under the composite $\rho\circ d_L$ is trivial. As noted in \cite[proof of Lemma 4.1.1]{asokSplittingVectorBundles2025}, by \cite[Theorem 2.3]{greenblattHomologyLocalCoefficients2006}, this map is the twisted Steenrod square $\Sq_L:x\mapsto\Sq(x)+w_1(L)\cup x$ where $w_1(L)$ is the Stiefel--Whitney class of $L$. By \cite[Théorème 4]{kahnConstructionClassesChern1987}, one has $w_1(L)=w_1$. We conclude $(c_1,c_2)$ lies in $\Im\varphi_2(X)$ if, and only if, the class $\Sq_{w_1}(w_2)\in\Hr^3(X(\Rb),\Zb/2)$ vanishes as required.
\end{proof}

\subsubsection*{Classification results}

In this subsection, we study analogues of \cite[Theorem 6.6]{asokCohomologicalClassificationVector2014} for smooth real affine varieties. To this end, we consider again the Postnikov tower $(\BGL_2^{(n)})_{n\geqslant 1}$ of $\BGL_2$. In contrast to the case of rank $2$ bundles on surfaces considered in Subsection \ref{subsection:rank_2_bundles_on_surfaces}, if $X$ is a smooth affine threefold over a field, the map $[X_+,\BGL_2]_{\Abb^1}\to[X_+,\BGL_2^{(2)}]_{\Abb^1}$ is surjective, but is generally not injective so rank $2$ bundles need not be classified by their determinant and Euler class, \emph{a fortiori} by their Chern classes. Our aim in this subsection is to bound the injectivity defect by hypotheses on the real locus.

Let $X$ be a smooth affine threefold over $\Rb$. Then the map $[X_+,\BGL_2]_{\Abb^1}\to[X_+,\BGL_2^{(3)}]_{\Abb^1}$ is indeed a bijection by \cite[Proposition 6.2]{asokCohomologicalClassificationVector2014}. Moreover, if $f:X_+\to\BGL_2^{(2)}$ is a morphism of motivic spaces inducing a torsor $\xi:X_+\to\Br\pi_1(\BGL_2)=\BGm$, hence a line bundle $\Lc$ on $X$, by composition with the $k_1$-invariant, then the set of lifts of $f$ along the map $\BGL_2^{(3)}\to\BGL_2^{(2)}$ up to $\Abb^1$-homotopy is in bijection with a quotient of $\Hr^3(X,\pi_2(\BGL_3)(\Lc))$. We then have:

\begin{prop}
Let $X$ be a smooth real affine threefold and suppose that $\Hr^3(X(\Rb),\Zb/2)=0$. Let $\Lc$ be a line bundle on $X$. Then the group $\Hr^3(X,\pi_3(\BGL_2)(\Lc))$ vanishes.
\end{prop}

\begin{proof}
As explained in \cite[proof of Theorem 6.6]{asokCohomologicalClassificationVector2014}, the sheaf $\pi_3(\BGL_2)(\Lc)$ is described in \cite[Theorem 3.3]{asokCohomologicalClassificationVector2014} by an exact sequence 
\begin{equation}\label{eq:desc_homotopy_sheaf_BGL}
0\to\Tbf_4'(\Lc)\to\pi_3(\BGL_2)(\Lc)\to\GWb_3^2(\Lc)\to 0
\end{equation}
where $\GWb_3^2(\Lc)$ is a twisted unramified higher Grothendieck--Witt sheaf and $\Tbf_4'(\Lc)$ is an extension of the form
\begin{equation}\label{eq:desc_tb4}
\Ibf^5(\Lc)\to\Tbf_4'(\Lc)\to\Sbf'_4(\Lc)\to 0
\end{equation}
where $\Sbf'_4$ is a quotient of $\Kbf_4^\Mr/12$ (in particular, the action on $\Sbf'_4$ is trivial so $\Sbf'_4(\Lc)\simeq\Sbf_4'$). By the results of \cite{jacobsonRealCohomologyPowers2017} (see \cite[Theorem 3.11]{hornbostelRealCycleClass2021}), there is an isomorphism $\Hr^3(X,\Ibf^5(\Lc))\cong\Hr^3(X(\Rb),\Zb(\Lc(\Rb)))$. As noted in the proof of Theorem \ref{theo:question_B_kucharz}, this group is a direct sum indexed by the compact connected components of $X(\Rb)$. Since $\Hr^3(X(\Rb),\Zb/2)=0$, there is no such component so $\Hr^3(X(\Rb),\Zb(\Lc(\Rb)))=\Hr^3(X,\Ibf^5(\Lc))$ vanishes. Moreover, by \cite[Corollary 4.3.6]{asokSplittingVectorBundles2025}, one has $\Hr^3(X,\Kbf_3^\Mr/12)\simeq\Hr^3(X(\Rb),\Zb/2)$ so the group $\Hr^3(X,\Kbf_3^\Mr/12)$ vanishes. Since $X$ is a threefold, it has cohomological dimension $\leqslant 3$ so the epimorphism $\Kbf_4^\Mr/12\to\Sbf'_4$ of sheaves induces an epimorphism $0=\Hr^3(X,\Kbf_4^\Mr/12)\to\Hr^3(X,\Sbf'_4)$ of top cohomology groups (see also \cite[Lemma 4.3.1]{asokSplittingVectorBundles2025}). For the same reason, if $\Jbf(\Lc)\subseteq\Tbf'_4(\Lc)$ is the image of the morphism $\Ibf^5(\Lc)\to\Tbf_4'(\Lc)$, one has $\Hr^3(X,\Jbf(\Lc))=0$. The cohomology long exact sequence associated with (\ref{eq:desc_tb4}) then yields an exact sequence \[\Hr^3(X,\Jbf(\Lc))=0\to\Hr^3(X,\Tbf'_4(\Lc))\to\Hr^3(X,\Sbf'_4(\Lc))=0\] and thus $\Hr^3(X,\Tbf'_4(\Lc))=0$. On the other hand, by \cite[Theorem 4.17]{asokCohomologicalClassificationVector2014}, the group $\Hr^3(X,\GWb_3^2(\Lc))$ is a quotient of $\Ch^3(X)$. By \cite[Theorem 3.2 (d)]{colliot-theleneZerocyclesCohomologyReal1996}, the group $\Ch^3(X)$ is isomorphic to $\Hr^3(X(\Rb),\Zb/2)$ so $\Ch^3(X)=0$ by assumption and thus its quotient $\Hr^3(X,\GWb_3^2(\Lc))$ is trivial. Finally, the cohomology long exact sequence associated to (\ref{eq:desc_homotopy_sheaf_BGL}) gives a short exact sequence \[0=\Hr^3(X,\Tbf'_4(\Lc))\to\Hr^3(X,\pi_3(\BGL_2)(\Lc))\to\Hr^3(X,\GWb_3^2(\Lc))=0\] so $\Hr^3(X,\pi_3(\BGL_2)(\Lc))=0$ as required.
\end{proof}

\begin{cor}\label{cor:euler_bijective_no_compact}
Let $X$ be a smooth real affine threefold. Suppose that $\Hr^3(X(\Rb),\Zb/2)=0$ and let $\Lc$ be a line bundle on $X$. Then the Euler class induces a bijection $e:\Vsc_2(X,\Lc)\xrightarrow{\simeq}\widetilde{\CH}^2(X,\Lc)$.
\end{cor}

\begin{proof}
As noted previously, since $X$ is a threefold, the canonical maps $\BGL_2\to\BGL_2^{(2)}$ and $\BGL_2\to\BGL_2^{(3)}$ of the Postnikov tower induce a surjective map $[X_+,\BGL_2]_{\Abb^1}\to[X_+,\BGL_2^{(2)}]_{\Abb^1}$ and a bijective map $[X_+,\BGL_2]_{\Abb^1}\to[X_+,\BGL_2^{(3)}]_{\Abb^1}$. Moreover, by the previous proposition, the fibres of the map $[X_+,\BGL_2^{(3)}]_{\Abb^1}\to[X_+,\BGL_2^{(2)}]_{\Abb^1}$ are quotients of sets of the form $\Hr^3(X,\pi_2(\BGL_3)(\xi))=*$ so they are singletons hence the map $[X_+,\BGL_2^{(3)}]_{\Abb^1}\to[X_+,\BGL_2^{(2)}]_{\Abb^1}$ is injective. Therefore the map $\Vsc_2(X)\to[X_+,\BGL_2^{(2)}]_{\Abb^1}$ is a bijection of sets over $\Pic(X)$ so it induces a bijection $\Vsc_2(X,\Lc)\to\widetilde{\CH}^2(X,\Lc)$ of fibres over the isomorphism class of $\Lc$.
\end{proof}

Thus the first Chern class and the Euler class provide a full cohomological classification of rank $2$ bundles on smooth real affine threefolds whose real locus has no compact connected component. To relate this to the bijectivity of $\varphi_2(X)$, we must make stronger assumptions on the real locus.

\begin{lem}\label{lem:chow_witt_equal_chow_small_real_locus}
Let $X$ be a smooth affine threefold and let $\Lc$ be a line bundle on $X$; suppose that $\Hr^3(X(\Rb),\Zb/2)=0$ and $\Hr^2(X(\Rb),\Zb(\Lc(\Rb)))=0$. Then the map $\widetilde{\CH}^2(X,\Lc)\to\CH^2(X)$ is injective.
\end{lem}

\begin{proof}
The exact sequence \[0\to\Ibf^3(\Lc)\to\Kbf_2^\MW(\Lc)\to\Kbf_2^\Mr\to 0\] of sheaves on $X$ induces an exact sequence \[\Hr^2(X,\Ibf^3(\Lc))\to\widetilde{\CH}^2(X,\Lc)\to\CH^2(X)\] so it suffices to prove that $\Hr^2(X,\Ibf^3(\Lc))=0$. By \cite[Proposition 3.1.1]{asokSplittingVectorBundles2025}, there is an isomorphism $\Hr^2(X,\Ibf^3(\Lc))\simeq\Hr^3(X(\Rb),\Zb/2)\oplus\Hr^2(X(\Rb),\Zb(\Lc(\Rb)))$ so the vanishing of $\Hr^2(X,\Ibf^3(\Lc))$ follows from our assumptions on $X(\Rb)$.
\end{proof}

\begin{theo}\label{theo:coh_classification_rank_2}
Let $X$ be a smooth real affine threefold and suppose that $\Hr^3(X(\Rb),\Zb/2)=0$ and $\Hr^2(X(\Rb),\Zb(\Lc(\Rb)))=0$ for every line bundle $\Lc$ on $X$. Then $\varphi_2(X)$ is bijective.
\end{theo}

\begin{proof}
If $\Lc$ is a line bundle on $X$, then the composite \[\Vsc_2(X,\Lc)\xrightarrow{e}\widetilde{\CH}^2(X,\Lc)\to\CH^2(X)\] is a composite of bijections by Corollary \ref{cor:euler_bijective_no_compact} and Lemma \ref{lem:chow_witt_equal_chow_small_real_locus}. Moreover, the composite is the second Chern class $c_2:\Vsc_2(X,\Lc)\hookrightarrow\Vsc_2(X)\to\CH^2(X)$ by \cite[Remark 7.22]{morelA1AlgebraicTopologyField2012}. Thus the second Chern class induces a bijection from each fibre of $\Vsc_2(X)$ over $\Pic(X)$ to $\CH^2(X)$. This is a reformulation of the bijectivity of $\varphi_2(X)$.
\end{proof}

\begin{rema}\label{rema:hypotheses_classification_rank_2}
Let $X$ be a smooth real affine threefold. There is an isomorphism $\Hr^2(X,\Ibf^3(\Lc))\simeq\Hr^3(X(\Rb),\Zb/2)\oplus\Hr^2(X(\Rb),\Zb(\Lc(\Rb)))$ for every line bundle $\Lc$ on $X$. Both factors $\Hr^3(X(\Rb),\Zb/2)$ and $\Hr^2(X(\Rb),\Zb(\Lc(\Rb)))$ could induce a defect of injectivity of the map $\widetilde{\CH}^2(X,\Lc)\to\CH^2(X)$, hence a failure of Theorem \ref{theo:coh_classification_rank_2} to hold. We observe here that both points of failure do in fact manifest.
\begin{itemize}
	\item Consider the real algebraic $2$-sphere $\Sr_\Rb^2$ and set $X=\Sr_\Rb^2\times\Abb^1$ (in particular, the real locus $X(\Rb)$ of $X$ has no compact connected component). Then $\widetilde{\gamma}^2:\widetilde{\CH}^2(X)\to\Hr^2(X(\Rb),\Zb)\cong\Hr^2(\Sr_\Rb^2,\Zb)=\Zb$ is an isomorphism, while $\CH^2(\Sr_\Rb^2)=\Zb/2$. Indeed this follows from the computations of Example \ref{exe:algebraic_vector_bundles_on_real_two_sphere} and from the $\Abb^1$-invariance of the cohomology theories involved.
	\item Let $U$ be the affine complement of a smooth hypersurface $X\subseteq\Pb_\Rb^3$ of degree $\delta\geqslant 4$ congruent to $2$ mod $4$ such that $X(\Rb)$ is empty and such that the restriction map $\Pic(\Pb_\Rb^3)\to\Pic(X)$ is surjective (see \cite[Proposition 3.2.2]{asokSplittingVectorBundles2025} for the existence of such hypersurfaces). By \cite[Lemma 3.2.8]{asokSplittingVectorBundles2025}, the restriction of the map $\alpha:\Hr^2(U,\Ibf^3)\to\widetilde{\CH}^2(U)$ to the summand $\Hr^3(U(\Rb),\Zb/2)$ is injective. However, the restriction of $\alpha$ to $\Hr^2(U(\Rb),\Zb)$ is trivial. Indeed there is a commutative square
\[\begin{tikzcd}
	{\Hr^2(\Pb_\Rb^3,\Ibf^3)} & {\Hr^2(\Pb_\Rb^3,\Ibf^3)} \\
	{\Hr^2(\Pb^3(\Rb),\Zb)} & {\Hr^2(U(\Rb),\Zb)}
	\arrow["{j^*}", from=1-1, to=1-2]
	\arrow["{\gamma^2}"', from=1-1, to=2-1]
	\arrow["{\gamma^2}", from=1-2, to=2-2]
	\arrow["{j(\Rb)^*}"', from=2-1, to=2-2]
\end{tikzcd}\]
induced by the open immersion $j:U\hookrightarrow\Pb_\Rb^3$ whose bottom horizontal morphism is an isomorphism because $X(\Rb)=\emptyset$ by choice of $X$ and whose left vertical map is an isomorphism by \cite[5.7 Theorem (a)]{hornbostelRealCycleClass2021}. In view of the commutative square
\[\begin{tikzcd}
	{\Hr^2(\Pb_\Rb^3,\Ibf^3)} & {\Hr^2(U,\Ibf^3)} \\
	{\widetilde{\CH}^2(\Pb_\Rb^3)} & {\widetilde{\CH}^2(U)}
	\arrow["{j^*}", from=1-1, to=1-2]
	\arrow[from=1-1, to=2-1]
	\arrow[from=1-2, to=2-2]
	\arrow["{j^*}"', from=2-1, to=2-2]
\end{tikzcd}\]
it suffices to prove that the map $\Hr^2(\Pb_\Rb^3,\Ibf^3)\to\widetilde{\CH}^2(\Pb_\Rb^3)$ is trivial. This is clear since $\Hr^2(\Pb_\Rb^3,\Ibf^3)\simeq\Hr^2(\Pb^3(\Rb),\Zb)=\Zb/2$ is $2$-torsion and $\widetilde{\CH}^2(\Pb_\Rb^3)=\Zb$ is torsion free (\cite[Corollary 11.8]{faselProjectiveBundleTheorem2013}). Thus the failure of the map $\widetilde{\CH}^2(U)\to\CH^2(U)$ to be injective indeed comes from the summand $\Hr^3(U(\Rb),\Zb/2)$ in $\Hr^2(U,\Ibf^3)$. We come back to this example in Proposition \ref{prop:ce_qu_A} below.
\end{itemize}
\vspace{-\topsep}\vspace{-\topsep}
\end{rema}

\begin{exe}
Let $X$ be a smooth real threefold such that $X(\Rb)$ is empty. Then by Theorem \ref{theo:coh_classification_rank_2}, the map $\varphi_2(X)$ is injective. In particular, cancellation holds for rank $2$ bundles: if $E$ and $F$ are rank~$2$ vector bundles on $X$ and if $E\oplus\Osc_X^n\simeq F\oplus\Osc_X^n$ for some $n$, then $E\simeq F$. Indeed, since Chern classes are stable, namely invariant under adding a trivial bundle, the bundles $E$ and $F$ have the same Chern classes, hence their isomorphism classes $\{E\}$ and $\{F\}$ have the same image under $\varphi_2(X)$. Since $\varphi_2(X)$ is injective, we conclude that $E$ and $F$ are isomorphic. As a particular case, if $E$ is a stably trivial rank~$2$ bundle, that is, if $E\oplus\Osc_X^n\simeq\Osc_X^2\oplus\Osc_X^n$ for some $n$, then $E\simeq\Osc_X^2$ is in fact trivial.

We can compare this to the results of \cite{syedStablyFreeModules2026}. In this paper, Syed consider smooth affine threefolds $U$ over $\Rb$ that further enjoy property P, which consists in the satisfaction of one of the following two assertions:
\begin{itemize}
	\item[(i)] the real locus $U(\Rb)$ of $U$ is empty;
	\item[(ii)] the intersection of the \emph{real} maximal ideals of $\Osc(U)$, namely those maximal ideals $\m$ such that $\Osc(U)/\m\simeq\Rb$, has height $\geqslant 1$.
\end{itemize}
But if $U$ is a smooth variety over $\Rb$ and $U(\Rb)$ is nonempty, then $U(\Rb)$ is (Zariski-)dense in $U$ (\cite[Proposition 1.1]{benoistHilberts17thProblem2017}; this well-known statement goes back to Artin) so it is not contained in a closed subset of dimension $<\dim(U)$ and thus the intersection of the real maximal ideals of $\Osc(U)$ has height $0$. Thus a \emph{smooth} variety $U$ with property P does not satisfy (ii), hence has empty real locus. Now Syed proves that if $U$ is a smooth variety satisfying property P or, equivalently, having empty real locus, then stably trivial rank $2$ vector bundles on $U$ are trivial if, and only if, a certain Hermitian $K$-theory group $\Wr_{\mathrm{SL}}(U)$ vanishes. This equivalence is most useful because, as mentioned in the introduction of \cite{syedStablyFreeModules2026}, as part of a cohomology theory, the group $\Wr_{\mathrm{SL}}(U)$ is as computable as can reasonably expected, so that the equivalence proven in \cite{syedStablyFreeModules2026} gives an accessible criterion for the triviality of stably trivial rank $2$ bundles on the varieties under consideration. But in fact, the argument of the previous paragraph shows that these equivalent assertions actually hold.
\end{exe}

\section{Rank 3 vector bundles}\label{sec:rank_3_bundles}

\subsection{Kucharz' theorem over general fields}

Our aim in this subsection is to prove the following theorem.

\begin{theo}[Fasel]\label{theo:Kucharz_over_a_field}
Let $k$ be a field. Let $X$ be a smooth affine threefold over $k$ and let $(c_i)_i\in\CH^1(X)\times\CH^2(X)\times\CH^3(X)$ be a triple. The following assertions are then equivalent.
\begin{itemize}
	\item[(i)] There exists a rank $3$ vector bundle $E$ on $X$ such that $c_i(E)=c_i$ for every $i$.
	\item[(ii)] Denoting by $\overline{c_i}$ the reduction mod $2$ of $c_i$ in $\Ch^i(X)$, one has $\overline{c_3}=\Sq_{\overline{c_1}}(\overline{c_2})$ in $\Ch^3(X)$.
\end{itemize}
\vspace{-\topsep}\vspace{-\topsep}
\end{theo}

We thank Jean Fasel for communicating to us the following argument.

\begin{proof}
The fact that (i) implies (ii) can be checked using the Adem relations (\cite[Section 11]{brosnanSteenrodOperationsChow2003}, \cite[Section 10]{voevodskyReducedPowerOperations2003}) as in \cite{hsiangWusFormulaSteenrod1963} (it can also be checked directly using the definition of Steenrod squares on Chow groups mod $2$).

Conversely, suppose that (ii) holds. Recall from \cite[§3]{grothendieckTheorieClassesChern1958} that there are group homomorphisms \[\pi_i:\CH^i(X)\to\Gra^i\Kr_0(X),\;c_i:\Gra^i\Kr_0(X)\to\CH^i(X)\] where the morphism $\pi_i$ carries the class of a codimension $i$ subvariety $Y$ to the $\Kr$-theory class of the coherent sheaf $\Osc_Y$, the morphism $c_i$ maps the class $[F]$ of a coherent sheaf $F$ on $X$ supported in codimension $\geqslant i$ to its $i$-th Chern class $c_i(F)$, and $c_i\circ\pi_i$ is multiplication by $(-1)^{i-1}(i-1)!$.

Since $c_2$ is surjective, there exists $\alpha\in\Fr^2\Kr_0(X)$ such that $c_2(\alpha)=c_2$. Since $\alpha\in\Fr^2\Kr_0(X)$, we see that $c_1(\alpha)=0$. Let further $L$ be a line bundle such that $c_1(L)=c_1$. Then by Whitney's formula, one has $c_1(\alpha+[L])=c_1(\alpha)+c_1(L)=c_1$ since $c_1(\alpha)=0$, and $c_2(\alpha\oplus[L])=c_2(\alpha)+c_1(\alpha)\cdot c_1(L)+c_2(L)$ where $c_1(\alpha)=0$ again and $c_2(L)=0$ since $L$ is a vector bundle of rank $1$ so $c_2(\alpha+[L])=c_2(\alpha)=c_2$. Since $X$ has dimension $3$, it follows from Serre's splitting theorem \cite[Théorème 1]{serreModulesProjectifsEspaces1957} that $\alpha+[L]\in\Kr_0(X)$ is of the form $[\Osc_X^r]+[E]-[\Osc_X^3]$ where $E$ has rank $3$ and $r=\rk(\alpha)+\rk(L)$ is the rank of $\alpha+[L]$; note that $\alpha$ has rank $0$ since $\alpha\in\Fr^2\Kr_0(X)\subseteq\Fr^1\Kr_0(X)=\Ker\rk$ so we obtain $\alpha+[L]=[E]-[\Osc_X^2]\in\Kr_0(X)$. In particular, by stability of Chern classes, one has $c_i(E)=c_i(\alpha+[L])=c_i$ for $i\in\{1,2\}$ by construction.

Now since (i) implies (ii), the relation \[\overline{c_3}(E)-\Sq_{\overline{c_1}(E)}(\overline{c_2}(E))=\overline{c_3}(E)-\Sq_{\overline{c_1}}(\overline{c_2})=0\] between the mod $2$ Chern classes of $E$ is satisfied, where $\overline{c_i}(E)$ is the image of $c_i(E)$ in $\Ch^i(X)$. On the other hand, by (ii), we also have $\overline{c_3}-\Sq_{\overline{c_1}}(\overline{c_2})=0$. It follows that $\overline{c_3}(E)-\overline{c_3}=0$ in $\Ch^3(X)$. In other words, there exists $b_3\in\CH^3(X)$ such that $2b_3=c_3(E)-c_3$. Now the composite \[\CH^3(X)\xrightarrow{\pi_3}\Gra^3\Kr_0(X)\xrightarrow{c_3}\CH^3(X)\] is multiplication by $-2$ so $c_3\circ\pi_3(b_3)=-2b_3=c_3-c_3(E)$ in $\CH^3(X)$. Since $X$ is a threefold, the group $\Fr^4\Kr_0(X)$ vanishes so we may view $\beta=\pi_3(b_3)$ as an element of $\Fr^3\Kr_0(X)$. Since $\beta\in\Fr^3\Kr_0(X)$, the equalities $c_1(\beta)=0$ and $c_2(\beta)=0$ hold as before. Thus using the Whitney formula again, it is easy to check that $[E]\oplus\beta$ has Chern classes $c_i([E]+\beta)=c_i$ for every $i$. Finally, again because of Serre's splitting theorem, there exists a rank $3$ bundle $E'$ on $X$ such that $[E]+\beta=[\Osc_X^r]+[E']-[\Osc_X^3]$ in $\Kr_0(X)$ where $r=\rk(E)+\rk(\beta)$. Since $\beta$ lies in $\Fr^3\Kr_0(X)$, its rank is zero so $r=\rk(E)=3$ and thus $[E]+\beta=[E']$. We conclude that $c_i(E')=c_i([E'])=c_i([E]+\beta)=c_i$ as required.
\end{proof}

\begin{exe}\label{exe:divisible_chow_group}
If $\Ch^3(X)=0$, that is, if $2\CH^3(X)=\CH^3(X)$, then assertion (ii) in Theorem \ref{theo:Kucharz_over_a_field} obviously holds so $\varphi_3(X)$ is surjective. The equality $\CH^3(X)=2\CH^3(X)$ holds if $k$ is algebraically closed since $\CH^d(X)$ is in fact divisible in this case (see, \emph{e.g.}, \cite[Lemma 1.2]{colliot-theleneZerocyclesCohomologyReal1996}).
\end{exe}

Thanks to this statement, we can give an efficient proof of Kucharz's theorem.

\begin{proof}[Proof of Kucharz's theorem]
Let $X$ be a smooth real affine threefold and let $(c_i)_i\in\CH^1(X)\times\CH^2(X)\times\CH^3(X)$; set $w_i=\overline{\gamma}^i(c_i)$. We note that \[\overline{\gamma}^3(\overline{c_3}-\Sq_{\overline{c_1}}(\overline{c_2}))=\overline{\gamma}^3(c_3)-\overline{\gamma}^3(\Sq_{\overline{c_1}}(\overline{c_2}))=w_3-\Sq_{w_1}(\overline{\gamma}^2(c_2))=w_3-\Sq_{w_1}(w_2).\] Since $X$ is affine, by \cite[Theorem 3.2 (d)]{colliot-theleneZerocyclesCohomologyReal1996}, the map $\overline{\gamma}^3:\Ch^3(X)\to\Hr^3(X(\Rb),\Zb/2)$ is injective so $w_3-\Sq_{w_1}(w_2)=0$ in $\Hr^3(X(\Rb),\Zb/2)$ if, and only if, one has $\overline{c_3}-\Sq_{\overline{c_1}}(\overline{c_2})=0$ in $\Ch^3(X)$ which is equivalent to the assertion that $(c_1,c_2,c_3)$ lies in $\Im\varphi_3(X)$ by Fasel's theorem above.
\end{proof}


\begin{rema}
It is possible to prove Theorem \ref{theo:question_B_kucharz} along the same lines as the proof of Fasel's result. More precisely, let $X$ be a smooth affine threefold over a field $k$ and let $(c_1,c_2)\in\CH^1(X)\times\CH^2(X)$, and suppose that $\Sq_{w_1}(w_2)=0\in\Hr^3(X(\Rb),\Zb/2)$, where $w_i=\overline{\gamma}^i(c_i)$. By \cite[Theorem 3.2 (d)]{colliot-theleneZerocyclesCohomologyReal1996}, the homomorphism $\overline{\gamma}^3$ is injective. Since $\Sq_{w_1}\circ\overline{\gamma}^2=\overline{\gamma}^3\circ\Sq_{c_1}$, the equality $\Sq_{w_1}(w_2)=0$ implies that $\Sq_{\overline{c_1}}(\overline{c_2})=0$ where $\overline{c_i}$ is the image of $c_i$ in $\Ch^i(X)=\CH^i(X)/2$. As in the proof of Theorem \ref{theo:Kucharz_over_a_field}, one constructs a rank $3$ vector bundle $E$ on $X$ such that $c_1(E)=c_1$ and $c_2(E)=c_2$. Then $\overline{c_3}(E)=\Sq_{\overline{c_1}}(\overline{c_2})=0$ in $\Ch^3(X)$ so $c_3(E)$ is of the form $2b_3$ where $b_3\in\CH^3(X)$. We conclude that $c_3(E)=-c_3(\beta)$ for some $\beta\in\Fr^3\Kr_0(X)$ and thus the Chern classes of $[E]+\beta$ are given by the triple $(c_1,c_2,0)$. By Serre's splitting theorem, there exists a rank $3$ vector bundle $E'$ on $X$ such that $[E]+\beta=[E']$. Now $c_3(E')=0$ by construction: assuming now that $k=\Rb$, since $X$ has odd dimension, it follows from \cite[Theorem 4.30]{bhatwadekarProjectiveModulesSmooth2006} that $E'$ splits off a free rank $1$ summand, namely is of the form $E'\simeq E''\oplus\Osc_X$ where $E''$ is a vector bundle of rank $2$ on $X$ (see also \cite[Remark 2.1.2, 3.]{asokSplittingVectorBundles2025}; from the point of view of \cite{asokSplittingVectorBundles2025}, the observation is that the Euler class of a topological vector bundle of top rank reduces to its top Stiefel--Whitney class over an odd-dimensional CW-complex). Then $c_i(E'')=c_i$ for $i\in\{1,2\}$.

Conversely, consider the map $\varphi_3:\BGL_3\to\prod_i\Kr(\Kbf_i^\Mr,i)$ induced by Chern classes and let $f:X_+\to\prod_i\Kr(\Kbf_i^\Mr,i)$ be a morphism of spaces inducing a triple $(c_i)_i\in\prod_i\CH^i(X)$. As mentioned in the introduction, analysis of the Moore--Postnikov tower of $\varphi_3$ shows that there is a unique obstruction $o(f)$ to lifting $f$ along $\varphi_3$ which lives in $\Ch^3(X)$. Fasel's theorem then gives strong evidence for an identification $o(f)=\overline{c_3}-\Sq_{\overline{c_1}}(\overline{c_2})$, at least up to a unit. It may be possible to prove it using Voevodsky's and Hoyois--Kelly--{\O}stv{\ae}r's description of cohomology operations on mod $2$ motivic cohomology (\cite{voevodskyMotivicEilenbergMacLaneSpaces2010,hoyoisMotivicSteenrodAlgebra2017}).
\end{rema}

\subsection{Question (A)}

In this subsection, we show that Question (A) has a \emph{negative} answer. This counter-example relies on the results of \cite{asokSplittingVectorBundles2025}. More precisely, recall the following theorem:

\begin{theo}[\protect{\cite[Theorem 3.2.1]{asokSplittingVectorBundles2025}}]\label{theo:secondary_obstruction_subtle}
Let $\delta\geqslant 4$ be an integer congruent to $2$ mod $4$ and let $X\subseteq\Pb_\Rb^3$ be a smooth surface such that $X(\Rb)$ is empty and the restriction $\Pic(\Pb_\Rb^3)\to\Pic(X)$ is surjective; set $U=\Pb_\Rb^3\setminus X$. Then the restriction $\widetilde{\CH}^2(\Pb_\Rb^3)\to\widetilde{\CH}^2(U)$ induces an isomorphism $\widetilde{\CH}^2(U)\cong\Zb/2\delta$ and there exists a rank $2$ bundle $E$ on $U$ such that $e(E)=\delta$, the Chern classes of $E$ vanish and the topological vector bundle $E(\Rb)$ is trivial.
\end{theo}

If $E$ is as in the above theorem, then by stability of Chern classes, the Chern classes of the rank $3$ bundle $E\oplus\Osc_U$ vanish. To produce a counter-example for Question (A), it now suffices to show that $E\oplus\Osc_U$ is not trivial. In fact:

\begin{prop}\label{prop:ce_qu_A}
There exist $U$ and $E$ as in Theorem \ref{theo:secondary_obstruction_subtle} such that $E$ is not stably trivial.
\end{prop}

\begin{proof}
Let us recall how $E$ is constructed in the proof of \cite[Theorem 3.2.1]{asokSplittingVectorBundles2025}. Let $X$ be a surface as in Theorem \ref{theo:secondary_obstruction_subtle} (such surfaces exist by \cite[Proposition 3.2.2]{asokSplittingVectorBundles2025} hence, ultimately, because of the Noether--Lefschetz theorem); set $U=\Pb_\Rb^3\setminus X$. There exists a vector bundle $F$ of rank $2$ on $\mathbb{P}_\Rb^3$ whose Chern classes are $c_1(F)=0$ and $c_2(F)=h^2$ where $h\in\CH^1(\mathbb{P}_\Rb^3)$ is a hyperplane section and underlying an alternating form $\varphi$. By \cite[Proposition 11]{faselChowWittGroups2009}, the class $\delta[(F,\varphi)_{|U}]\in\GW^2(U)$ is of the form $[(E,\psi)]+(\delta-1)[\Hr(\Osc_U)]$ where $E$ is a rank $2$ vector bundle on $U$ and $\psi$ is an alternating form on $E$, and $\Hr(\Osc_U)$ is the hyperpolic alternating form on $\Osc_U$; then $E$ has the required properties.

Recall the projective bundle formula \cite[Theorem 8.5]{weibelKbookIntroductionAlgebraic2013}: there is an isomorphism $\Kr_0(\Pb_\Rb^3)\cong\Zb[t]/\langle(1-t)^4\rangle$ of rings where $t=[\Osc_{\Pb_\Rb^3}(-1)]\in\Kr_0(\Pb_\Rb^3)$ (from now on, for ease of notation, we omit the index $\Pb_\Rb^3$ where applicable). In other words, the abelian group $\Kr_0(\Pb_\Rb^3)$ is freely generated by $([\Osc(-n)])_{0\leqslant n\leqslant 3}$ and the relation 
\begin{equation}\label{eq:relation_K_theory_proj}
[\Osc(-4)]-4[\Osc(-3)]+6[\Osc(-2)]-4[\Osc(-1)]+[\Osc]=0
\end{equation}
holds in $\Kr_0(\Pb_\Rb^3)$. From it, we can deduce the $\Kr$-theory class of $\Osc(n)$ for any $n\in\Zb$, using the product structure on $\Kr_0(\mathbb{P}_\Rb^3)$ induced by the tensor product of vector bundles. For instance, this yields \[[\Osc(1)]=-[\Osc(-3)]+4[\Osc(-2)]-6[\Osc(-1)]+4[\Osc].\] It will be convenient to represent the elements of $\Kr_0(\Pb_\Rb^3)$ as column and row vectors with respect to the ordered basis $([\Osc],[\Osc(-1)],[\Osc(-2)],[\Osc(-3)])$, so that \[a[\Osc]+b[\Osc(-1)]+c[\Osc(-2)]+d[\Osc(-3)]=\begin{bmatrix}
a \\
b \\
c \\
d \\
\end{bmatrix}=(a,b,c,d).
\] Thus $[\Osc(1)]=(4,-6,4,-1)$.

Now the vector bundle $F$ on $\mathbb{P}_\Rb^3$ sits in an exact sequence \[0\to F\to G\to\Osc(1)\to 0\] where $G$ sits in an exact sequence \[0\to\Osc(-1)\to\Osc^4\to G\to 0\] of vector bundles on $\mathbb{P}_\Rb^3$. Thus there are equalities $[G]=4[\Osc]-[\Osc(-1)]=(4,-1,0,0)$ and $[F]=[G]-[\Osc(1)]$ in $\Kr_0(\Pb_\Rb^3)$, so that \[[F]=(0,5,-4,1)\in\Kr_0(\mathbb{P}_\Rb^3).\] Therefore $[E]=\delta[F_{|U}]-2(\delta-1)[\Osc_U]\in\Kr_0(U)$ is the restriction of $(-2(\delta-1),5\delta,-4\delta,\delta)$ living in $\Kr_0(\Pb_\Rb^3)$. Since $U$ is affine, the bundle $E$ is stably trivial if, and only if, its $\Kr$-theory class satisfies $[E]=2[\Osc_U]$. Therefore our aim is to prove that the class $(-2\delta,5\delta,-4\delta,\delta)\in\Kr_0(\Pb_\Rb^3)$ does not restrict to the zero class in $\Kr_0(U)$.

From now on, we take $\delta=6$ hence $(-2\delta,5\delta,-4\delta,\delta)=(-12,30,-24,6)$. Since $X$ is smooth, so that its $\Kr$-theory coincides with its $\Gr$-theory, the inclusion $i:X\hookrightarrow\Pb_\Rb^3$ and the open immersion $j:U\hookrightarrow\Pb_\Rb^3$ induce an exact sequence \[\Kr_0(X)\xrightarrow{i_*}\Kr_0(\Pb_\Rb^3)\xrightarrow{j^*}\Kr_0(U)\to 0.\] Therefore to prove that the class $(-12,30,-24,6)\in\Kr_0(\Pb_\Rb^3)$ does not map to $0$ under $j^*$, it suffices to show that it does not lie in the image of $i_*$.

To do this, we consider again the filtration $\Fr^\bullet\Kr_0(X)$ given by codimension of the support. The Chow groups of $X$ are given by \[\CH^0(X)=\Zb\cdot[X],\;\CH^1(X)=\Zb\cdot(i^*h)\] and $\CH^2(X)$ is generated by the classes of closed points. Note that all such points of $X$ have residue field isomorphic to $\Cb$ since $X(\Rb)$ is empty by choice of $X$. The map $\CH^p(X)\to\Fr^p\Kr_0(X)/\Fr^{p+1}\Kr_0(X)$ carrying a subvariety $Y$ to the $\Kr$-theory class $[\Osc_Y]$ of the coherent sheaf $\Osc_Y$ is surjective, and $\Fr^p\Kr_0(X)=0$ for all $p\geqslant 3$ since $X$ is a surface. We conclude that $\Kr_0(X)$ is generated by the classes $[\Osc_X]$ and $[i^*\Osc(-1)]$ (recall that we omit the index $\Pb_\Rb^3$ when possible) together with the classes of closed points in $X$. Now there is an exact sequence \[0\to\Osc(-6)\to\Osc\to i_*\Osc_X\to 0\] of sheaves on $\Pb_\Rb^3$. This yields equalities \[[i_*\Osc_X]=[\Osc]-[\Osc(-6)],\;[i_*i^*\Osc(-1)]=[\Osc(-1)]-[\Osc(-7)]\] in $\Kr_0(\Pb_\Rb^3)$. Multiplying (\ref{eq:relation_K_theory_proj}) by $[\Osc(-1)]$ yields \[[\Osc(-5)]=4[\Osc(-4)]+(-6[\Osc(-3)]+4[\Osc(-2)]-[\Osc(-1)])=4\begin{bmatrix}
-1 \\
4 \\
-6 \\
4
\end{bmatrix}+\begin{bmatrix}
0 \\
-1 \\
4 \\
-6
\end{bmatrix}=\begin{bmatrix}
-4 \\
15 \\
-20 \\
10
\end{bmatrix}.
\] Then \[[\Osc(-6)]=4[\Osc(-5)]-6[\Osc(-4)]+(4[\Osc(-3)]-[\Osc(-2)])=
4\begin{bmatrix}
-4 \\
15 \\
-20 \\
10
\end{bmatrix}+(-6)\begin{bmatrix}
-1 \\
4 \\
-6 \\
4
\end{bmatrix}+\begin{bmatrix}
0 \\
0 \\
-1 \\
4
\end{bmatrix}\] so that $[\Osc(-6)]=(-10,36,-45,20)$. Finally, the same method yields
\[
[\Osc(-7)] = 4\begin{bmatrix}
-10 \\
36 \\
-45 \\
20
\end{bmatrix}+(-6)\begin{bmatrix}
-4 \\
15 \\
-20 \\
10
\end{bmatrix}+4\begin{bmatrix}
-1 \\
4 \\
-6 \\
4
\end{bmatrix}+\begin{bmatrix}
0 \\
0 \\
0 \\
-1
\end{bmatrix}=\begin{bmatrix}
-20 \\
70 \\
-84 \\
35
\end{bmatrix}.
\] We conclude that \[i_*[\Osc_X]=[\Osc]-[\Osc(-6)]=
\begin{bmatrix}
1 \\
0 \\
0 \\
0
\end{bmatrix}-
\begin{bmatrix}
-10 \\
36 \\
-45 \\
20
\end{bmatrix}=\begin{bmatrix}
11 \\
-36 \\
45 \\
-20
\end{bmatrix},\] \[i_*[i^*\Osc(-1)]=[\Osc(-1)]-[\Osc(-7)]=\begin{bmatrix}
0 \\
1 \\
0 \\
0
\end{bmatrix}-
\begin{bmatrix}
-20 \\
70 \\
-84 \\
35
\end{bmatrix}=
\begin{bmatrix}
20 \\
-69 \\
84 \\
-35
\end{bmatrix}.
\] Finally it remains to compute the image of classes of closed points. To do so, it suffices to compute the classes of complex closed points in $\Kr_0(\Pb_\Rb^3)$. If $x$ is such a point, then it defines a rational point in $\Pb_\Cb^3$ whose class in $\Kr_0(\Pb_\Cb^3)$ is $([\Osc_{\Pb_\Cb^3}]-[\Osc_{\Pb_\Cb^3}(-1)])^3$. The class of $x$ in $\Kr_0(\Pb_\Rb^3)$ is then given by $p_*(([\Osc_{\Pb_\Cb^3}]-[\Osc_{\Pb_\Cb^3}(-1)])^3)$ where $p$ is the finite étale degree $2$ projection $\Pb_\Cb^3\to\Pb_\Rb^3$. Since $([\Osc_{\Pb_\Cb^3}]-[\Osc_{\Pb_\Cb^3}(-1)])^3=p^*(([\Osc]-[\Osc(-1)])^3)$ and since $p_*p^*$ is multiplication by $2$, we conclude that the class of any complex point $x\in\Pb_\Rb^3(\Cb)$ in $\Kr_0(\Pb_\Rb^3)$ is given by \[2([\Osc]-3[\Osc(-1)]+3[\Osc(-2)]-[\Osc(-3)])=(2,-6,6,-2).\]

Finally we obtain that $\Im i_*$ is the subgroup of $\Kr_0(\Pb_\Rb^3)$ generated by \[(11,-36,45,-20),(20,-69,84,-35),(2,-6,6,-2).\] We now wish to show that $(-12,30,-24,6)$ does not lie in this subgroup. Assume by contradiction that \[a
\begin{bmatrix}
2 \\
-6 \\
6 \\
-2
\end{bmatrix}+b\begin{bmatrix}
20 \\
-69 \\
84 \\
-35
\end{bmatrix}+c\begin{bmatrix}
11 \\
-36 \\
45 \\
-20
\end{bmatrix}=\begin{bmatrix}
-12 \\
30 \\
-24 \\
6
\end{bmatrix}
\] where $(a,b,c)\in\Zb^3$. Thus $(a,b,c)$ is a solution of the linear system
\[\left\{
\begin{array}{rcl}
2a+20b+11c  &=& -12 \\
-6a-69b-36c &=& 30 \\
6a+84b+45c  &=& -24 \\
-2a-35b-20c &=& 6
\end{array}\right.
\] Multiplying the first equation by $3$ yields \[6a+60b+33c=-36.\] Substracting this equation from the third, we see that \[(6a+84b+45c)-(6a+60b+33c)=-24-(-36)=12\] thus \[24b+12c=12.\] Dividing this equation by $12$ yields $2b+c=1$ so $c=1-2b$ is odd. On the other hand, reducing mod $2$ the equality $2a+20b+11c=-12$ shows that $c$ is even: contradiction.
\end{proof}

\begin{rema}
We make a few comments about the above result.
\begin{enumerate}
	\item Since $U(\Rb)$ is compact and connected, we do not know of any general argument that would imply that the stabilisation map $\Vsc_3(U)\to\widetilde{\Kr}_0(U)$ is injective (see \cite{banerjeeSuslinsCancellationConjecture2026} for results in this direction) so there may exist rank $3$ bundles on $U$ that are stably trivial but not trivial, though, according to \cite[Theorem 4.15]{dasOrbitSpacesUnimodular2018}, there is at most one such bundle up to isomorphism. For instance, it is known (\cite[Theorem 4.16]{dasOrbitSpacesUnimodular2018}) that if $X$ is a real algebraic sphere of odd dimension $d\notin\{1,3,7\}$, the set of isomorphism classes of stably free rank $d$ vector bundles on $X$ is in bijection with $\Zb/2$, the nontrivial vector bundle being the tangent bundle (note however that all vector bundles over the real algebraic sphere of dimension $3$ are free by \cite{faselProjectiveModulesReal2011}).
	\item The variety $U$ is evidently rational even over $\Rb$ in Proposition \ref{prop:ce_qu_A}. Consequently, the assumption that the real locus has no compact connected component cannot be removed in \cite[p. 215, second-to-last paragraph]{kucharzAlgebraicCyclesVector1988}. As we shall see below (Theorem \ref{theo:classification_by_chern_classes}), this topological assumption is in fact the crucial one.
	\item To our knowledge, Proposition \ref{prop:ce_qu_A} gives the first example of a vector bundle on a smooth real affine threefold with trivial Chern classes that is not stably trivial. The examples exhibited in \cite{asokSplittingVectorBundles2025} appear to be rather subtle and we do not use the full strength of \cite[Theorem 3.2.1]{asokSplittingVectorBundles2025} in the above proof: indeed, we do not use that in Theorem \ref{theo:secondary_obstruction_subtle}, the bundle $E$ can be chosen so that $E(\Rb)$ is trivial. It would be interesting to construct a rank $3$ bundle $E$ on a smooth affine threefold $X$ with trivial Chern classes but such that the (reduced) real $\Kr$-theory class $[E(\Rb)]\in\widetilde{\mathrm{KO}}(X(\Rb))$ is nonzero (or to prove that such bundles do not exist); this would imply that the reduced $\Kr$-theory class $[E]$ is nonzero since it maps to $[E(\Rb)]$ under the real realisation map $\widetilde{\Kr}_0(X)\to\widetilde{\mathrm{KO}}(X(\Rb))$.
\end{enumerate}
\vspace{-\topsep}\vspace{-\topsep}
\end{rema}

\subsection{The classification of rank $3$ bundles on smooth affine threefolds with small real locus}

As mentioned in the introduction, the following theorem, whose $\Kr$-theoretic variant was suggested to us by Fasel, fits in the theme developed in \cite{asokSplittingVectorBundles2025} and especially in \cite{banerjeeSuslinsCancellationConjecture2026} as well as in Theorem \ref{theo:coh_classification_rank_2} according to which the theory of vector bundles of smooth real affine varieties of suitably (cohomologically) small real locus mirrors the same theory of smooth affine varieties over $\Cb$ (or more generally over an algebraically closed base field).

\begin{theo}\label{theo:classification_by_chern_classes}
Let $X$ be a smooth real affine threefold and suppose that $\Hr^3(X(\Rb),\Zb/2)=0$. The map $\varphi_3(X)$ is then bijective.
\end{theo}

\begin{proof}
If $(c_i)_i\in\prod_i\CH^i(X)$, then setting $w_i=\overline{\gamma}^i(c_i)$, the relation $w_3=\Sq_{w_1}(w_2)$ between elements of $\Hr^3(X(\Rb),\Zb/2)=0$ is automatically satisfied. By Kucharz's theorem, this guarantees that $\varphi_3(X)$ is surjective.

We now show that $\varphi_3(X)$ is injective. Let $E$ and $F$ be rank $3$ vector bundles on $X$ such that $c_i(E)=c_i(F)$ for every $i$. We have to show that $E$ and $F$ are isomorphic. Since $\Hr^3(X(\Rb),\Zb/2)$ is trivial and as $E$ has rank $3=\dim X$, by \cite[§4.3]{banerjeeSuslinsCancellationConjecture2026}, the bundle $E$ is cancellative so to prove this, it suffices to show that the $\Kr$-theory classes of $E$ and $F$ agree. Note that $E$ and $F$ have rank $3$ so $[E]-[F]\in\Fr^1\Kr_0(X)=\Ker\rk$. Next recall the homomorphisms \[\pi_i:\CH^i(X)\to\Gra^i\Kr_0(X),\;c_i:\Gra^i\Kr_0(X)\to\CH^i(X)\] from the proof of Theorem \ref{theo:Kucharz_over_a_field}. If $i\leqslant 2$, the morphism $\pi_i$ is surjective and $c_i\circ\pi_i$ is multiplication by $(-1)^{i-1}(i-1)!=(-1)^{i-1}$ which is injective. Thus $\pi_i$ is an isomorphism; since the composite $c_i\circ\pi_i$ is an isomorphism, the morphism $c_i$ is then also an isomorphism. In particular, since $c_i(E)=c_i(F)$ for $i\leqslant 2$, we see that $[E]-[F]$ in fact lies in $\Fr^3\Kr_0(X)$. Since $X$ has dimension $3$, the group $\Fr^4\Kr_0(X)$ is trivial so we may write $[E]-[F]=\pi_3(\alpha)\in\Kr_0(X)$ where $\alpha\in\CH^3(X)$. We note that $c_3([E]-[F])=0$. Indeed let $F'$ be a vector bundle on $X$ such that $F\oplus F'$ is trivial; then by stability of Chern classes, one has \[
\begin{array}{rcl}
c_3([E]-[F]) &=& c_3([E\oplus F']-[F\oplus F']) \\
             &=& c_3(E)+c_2(E)\cdot c_1(F')+c_1(E)\cdot c_2(F')+c_3(F') \\
             &=& c_3(F)+c_2(F)\cdot c_1(F')+c_1(F)\cdot c_2(F')+c_3(F')
\end{array}
\] as $E$ and $F$ have the same Chern classes by assumption. The last term is $c_3(F\oplus F')$ which vanishes since $F\oplus F'$ is free hence $c_3([E]-[F])=0$. Since $c_3\circ\pi_3$ is multiplication by $2$, we then have $0=c_3([E]-[F])=c_3(\pi_3(\alpha))=2\alpha$ so the class $\alpha\in\CH^3(X)$ is $2$-torsion. But since $X$ is affine of dimension~$3$, by \cite[Theorem 1.6 (a)]{colliot-theleneZerocyclesCohomologyReal1996}, the torsion subgroup of $\CH^3(X)$ is isomorphic to $\Hr^3(X(\Rb),\Zb/2)$. By our assumption on $X(\Rb)$, this means that $\CH^3(X)$ is torsion free so that in fact $\alpha=0$ in $\CH^3(X)$. It follows that $\pi_3(\alpha)=[E]-[F]=0$ in $\Kr_0(X)$, as required.
\end{proof}

{\begin{footnotesize}
\printbibliography
\end{footnotesize}}

\end{document}